\documentclass[11pt]{amsart}
\usepackage{fullpage,verbatim,amssymb}
\usepackage{hyperref}
\usepackage[usenames,dvipsnames]{color}
\usepackage{enumitem, soul}

\makeatletter
\let\@@pmod\mod
\DeclareRobustCommand{\mod}{\@ifstar\@pmods\@@pmod}
\def\@pmods#1{\mkern4mu({\operator@font mod}\mkern 6mu#1)}
\makeatother

\definecolor{blue}{rgb}{0,0,1}
\definecolor{red}{rgb}{1,0,0}
\definecolor{green}{rgb}{0,.6,.2}
\definecolor{purple}{rgb}{1,0,1}

\long\def\red#1\endred{\textcolor{red}{#1}}
\long\def\blue#1\endblue{\textcolor{blue}{#1}}
\long\def\purple#1\endpurple{\textcolor{purple}{ #1}}
\long\def\green#1\endgreen{\textcolor{green}{#1}}

\newcommand{\sm}{\left(\begin{smallmatrix}}
\newcommand{\esm}{\end{smallmatrix}\right)}
\newcommand{\bpm}{\begin{pmatrix}}
\newcommand{\ebpm}{\end{pmatrix}}

\newcommand{\Z}{\mathbb{Z}}

\let\Re\relax
\DeclareMathOperator{\Re}{Re}

\DeclareMathOperator{\ord}{ord}
\DeclareMathOperator{\SL}{SL}

\newtheorem{theorem}{Theorem}
\newtheorem{lemma}[theorem]{Lemma}
\newtheorem{proposition}[theorem]{Proposition}
\newtheorem{corollary}[theorem]{Corollary}

\theoremstyle{remark}

\newtheorem{remark}[theorem]{Remark}
\newtheorem{example}[theorem]{Example}
\numberwithin{theorem}{section}
\numberwithin{equation}{section}

\title{On the functional equation of twisted Ruelle zeta function and Fried's conjecture
}
\author{Jay Jorgenson}
\author{Min Lee}
\author{Lejla Smajlovic}
\thanks{The second author is supported by a Royal Society University Research Fellowship.}

\begin{document}
\maketitle

\begin{abstract}\noindent
Let $M$ be a finite volume hyperbolic Riemann surface of arbitrary genus, number of cusps,
and elliptic points.  Let $\chi$ be an arbitrary multiplier system associated to $M$,
composed of a unitary representation of the fundamental group of $M$ and an admissible weight.
We develop a functional equation for the product $R(s;\chi)\varphi(s;\chi)$ where $R(s;\chi)$ is
the (twisted by $\chi$) Ruelle zeta function and $\varphi(s;\chi)$ is the determinant of the scattering matrix.  The
factor in the functional equation is expressed solely in terms of the function $\sin(s)$
and its translations, hence is a first-order meromorphic function.
From this expression, we determine the order of the divisor $\ord_R(0)$ of $R(s;\chi)$ at $s=0$ as well
as its lead term in the Laurent expansion, up to $\pm$ sign.  When combined with results of Kitano and of Yamaguchi,
we prove further instances of the Fried conjecture, which asserts in general terms that (higher-dimensional) Reidemeister
torsion is equal to $\vert R(0;\chi)\vert^{\pm 1}$, depending on certain
normalizations.  More specifically, in \cite{Ki96} Kitano evaluates Reidemeister torsion for certain Seifert fibered
spaces associated to $\SL(n,\mathbb{C})$ irreducible acyclic representations $\chi$;  we show that
one has the same value for $\vert R(0;\chi)\vert^{-1}$.  In \cite{Yamaguchi17} Yamaguchi computes higher-dimensional
Reidermeister torsion for certain Seifert fibered spaces and even dimensional non-unitary representations of the fundamental group,
and then proves in \cite{Yamaguchi22} that the Fried conjecture holds for some of the cases considered in \cite{Yamaguchi17}.
Our computations complete the proof of the Fried conjecture for all settings addressed in \cite{Yamaguchi17}
and provides an alternate proof of the main result in \cite{Yamaguchi22}.
Finally, we compute the limit $\lim_{s\to 0}|s^{-\ord_R(0)}R(s;\chi)|$ in some cases when $M$ is the non-compact quotient of the hyperbolic upper
half plane by a congruence group and $\chi$ is a Dirichlet character. We show that this number can be expressed in terms of special values of Dirichlet $L$-functions and other number-theoretic quantities.
\end{abstract}

\section{Introduction}

\subsection{R-torsion and analytic torsion}
Reidemeister torsion or Reidemeister-Franz torsion, often simply referred to as R-torsion, is a topological invariant
of unquestionable significance.  The R-torsion of a $3$-manifold $M$, which is denoted by $\tau(M)$,
was first defined in
\cite{Re35} where $\tau(M)$ was used to prove the existence of
certain lens spaces which are homotopically equivalent but not homeomorphic.
It has been stated that \cite{Re35} was the historical beginning of geometric topology.

The subsequent development and study of R-torsion is itself a vast area of mathematics, including
higher dimensional definitions of $\tau(M)$ beginning with Franz \cite{Fr35} and de Rham \cite{deR64}.
Of the many aspects and uses of R-torsion perhaps none are more striking than
its presence in Milnor's counterexamples for the \it Hauptvermutung; \rm see \cite{Mi61} as well
as \cite{Co73}, \cite{Ra01} and \cite{Ra96}.

Let $M$ be a compact, smooth Riemannian manifold.  In \cite{RS71}, Ray and Singer showed that R-torsion $\tau(M)$ associated to a triangulation $\mathit{Tri}(M)$ of $M$ can be realized as a special value of a certain
linear combination of graph-theoretic spectral zeta functions associated to combinatorical
Laplacians corresponding to $\mathit{Tri}(M)$; see also section 6 of
of \cite{Mu78}.  Going further, Ray and Singer defined a spectral-theoretic invariant $T(M)$
associated to the Riemannian manifold $M$, which they called
as \it analytic torsion\rm, by using a similarly defined linear combination of spectral zeta functions associated
to the Riemannian Laplacian on $M$.  Furthermore, Ray and Singer conjectured that
$\tau(M) = T(M)$, meaning that the combinatorial Laplacian invariant $\tau(M)$ is equal
to the Riemannian Laplacian invariant $T(M)$.   The proof of this
conjecture was given by Cheeger \cite{Ch79} and by M\"uller \cite{Mu78}, and the result that
$\tau(M) = T(M)$ is known as the Cheeger-M\"uller
Theorem.

Of course, the aforementioned articles defined and studied R-torsion and analytic torsion in its nascent form.  In subsequent
years, many authors have obtained and developed vast generalizations of R-torsion and analytic torsion, and in
many instances the analogue of the Cheeger-M\"uller theorem has been shown to be true. Furthermore, analytic
torsion has played a role in other mathematical fields, such as algebraic geometry and arithmetic geometry.
We refer to the monograph \cite{CFM16} for further reading and guide to the existing literature.

As a separate remark, let us point out that in
\cite{RS73} the authors defined another invariant which they also called ``analytic torsion'', this time associated to
a complex Hermitian manifold.  The careful discussion given in
\cite{DY22} suggest that the invariant defined in \cite{RS73} perhaps should be
called \it holomorphic torsion \rm in order to distinguish from the analytic torsion from \cite{RS71};
for now, let us use simply use the phrase ``holomorphic version of analytic torsion'' in referring to the
invariant from \cite{RS73}.
With this terminology, we note that holomorphic version of analytic torsion from \cite{RS73} is defined through special values of spectral
zeta functions in analogy
with analytic torsion from \cite{RS71}.  The holomorphic version of analytic torsion has been used to construct modular forms on certain moduli
spaces and is a vital ingredient in Arakelov theory, specifically in the definition of Quillen metrics.

\subsection{Ruelle zeta function and the Fried conjecture}
For simplicity, let us now assume that $M$ is a compact and oriented hyperbolic manifold, which is the setting considered in
\cite{Fr86b}. Let $\chi$ be a representation of the fundamental group $\pi_{1}(M)$ of $M$.
In this context, there is a Ruelle zeta function which is defined for $\text{\rm Re}(s)$ sufficiently large
by
$$
R(s;\chi) := \prod\limits_{\gamma}\text{\rm det}\left(I - \chi(\gamma)e^{-s \ell(\gamma)}\right),
$$
where the product is over inequivalent classes of prime and closed hyperbolic classes of the fundamental group of $M$
and $\ell(\gamma)$ is the length of the shortest closed geodesic in the homotopy class determined by $\gamma$.
With the data $M$ and $\chi$, one has a well-defined R-torsion $\tau(M,\chi)$ and analytic torsion $T(M,\chi)$, and the analogous
Cheeger-M\"uller theorem has been proved, namely that $\tau(M,\chi) = T(M,\chi)$; see, for example, \cite{Mu93}
and references therein as well as the original articles
\cite{Ch79} and \cite{Mu78}.
The main result in \cite{Fr86b} is that the analytic torsion $T(M,\chi)$ of $M$ and $\chi$,
hence the R-torsion $\tau(M,\chi)$ via the Cheeger-M\"uller theorem, is equal to $\vert R(0;\chi)\vert$, or its inverse
depending on the dimension of $M$.  In other words, the topological or graph-theoretic invariant $\tau(M,\chi)$, which is known
to equal the spectral invariant $T(M,\chi)$, is given succinctly in terms of the length-spectral invariant $\vert R(0;\chi)\vert$.

At this time, one generally uses the term \it Fried conjecture \rm to mean an assertion that analytic torsion $T(M)$, or one of its many generalizations,
is easily expressed in terms of a special value of a corresponding Ruelle  zeta function $R(s)$, often when $s=0$.

By their definition, spectral zeta functions are like Dirichlet series by which they are expressed as convergent
series in a right-half-plane.  However, Ruelle and Selberg zeta functions are very akin to number theoretic zeta functions
because they can be defined as a type of Euler product in a right-half-plane.  With this in mind, one can view the results
from \cite{Fr86b} as connecting R-torsion to special values of a type of number theoretic zeta function, namely $R(s;\chi)$.

On the other hand, in \cite{KN22}, the R-torsion is proved to be an algebraic integer for some special 3-manifolds. Thus, Fried's conjecture in this context implies that the special value of the Ruelle zeta function at $s=0$  is an algebraic integer.

\subsection{The purpose of this article}
In this article we study the Ruelle zeta function $R(s;\chi)$ associated to a finite volume hyperbolic
Riemann orbifold $M$ associated to an arbitrary unitary representation $\chi$ and general multiplier
system.   When combined with computations by Kitano \cite{Ki96} and by Yamaguchi \cite{Yamaguchi17} we prove
new instances of the Fried conjecture, in particular an extension of the results from \cite{Yamaguchi22}.
In further detail, our main results are as follows.

First we derive a functional equation for the Ruelle zeta function.  Let $\varphi(s;\chi)$ be the
determinant of the scattering matrix associated to the situation in hand; if $M$ is co-compact, then take
$\varphi(s;\chi) \equiv 1$.   Then for all $s \in \mathbb{C}$, we have that
\[R(s;\chi)\varphi(s;\chi) = \left(R(-s;\chi)\varphi(-s;\chi)\right)^{-1}H(s)\]
where $H(s)$ is an order one meromorphic function which is explicitly and completely given in terms of
the trigonometric function $\sin(s)$ and its translations; see Proposition \ref{prop: funct eq R zeta}.
As a corollary, in the instance when $M$ is co-compact, we derive the following functional equation relating the Ruelle zeta functions twisted by the multiplier system $\chi_{k}$ and $\chi_{-k}=\overline{\chi_k}$ of admissible weights $2k$ and $-2k$ respectively:
\[R(s;\chi_{-k})R(-s;\chi_{-k})=R(s;\chi_{k})R(-s;\chi_{k}).\]
Fried in \cite{Fr86c} studied the functional equation of the Ruelle zeta function $R(s; \chi)$ in the setting when $M$ is compact
and $\chi$ is unitary; however, the analysis is complicated by the classical though somewhat intractable
form of the functional equation of the Selberg zeta function.  More recent considerations of the
functional equation of $R(s;\chi)$ include results from \cite{FS23}, which assumes that $M$ is compact,
and \cite{Teo20}, which assumes that $\chi$ is trivial.  In both \cite{FS23} and \cite{Teo20}, the authors
prove a form of the functional equation which, as in Proposition \ref{prop: funct eq R zeta} below,
expresses $H(s)$ in terms of the trigonometric function $\sin(s)$, thus showing that $H(s)$ has order one.
Our method of proof employs various computations from  regularized product
expressions for the Selberg zeta function.

Second, we determine the order of the divisor of $R(s;\chi)$ at $s=0$.  We
prove that the order depends on the topological signature of $M$, the weight $2k$ of the multiplier
system, and data associated to the representation $\chi$.  In particular, if the weight is an even integer, meaning
$k \in \mathbb{Z}$, then the order of $R(s;\chi)$ at $s=0$, which we denote by $\ord_R(0;\chi)$, is given by
$$
\ord_R(0;\chi)= m(2g-2+\rho+\tau) - \tau_0-\tilde{\tau}_0-n_0,
$$
where $\tau_0$ and
$\tilde{\tau}_0$ are parabolic and elliptic orders of singularity of $\chi$, respectively, and $n_0$ is the order of divisor at $s=0$ of
the Dirichlet series portion of the scattering determinant.  We refer to Sections \ref{sec:notation},  \ref{sec: mult systems} and
\ref{ss:scattering} below for precise definitions and notation.
When the weight is not an even integer, meaning $k \notin \mathbb{Z}$, then the order of divisor of $R(s;\chi)$ at $s=0$ is $-n_0$.
In particular, if $M$
is compact and $k\notin \mathbb{Z}$, then $R(s;\chi)$ is non-vanishing at $s=0$.  These results are given
in Theorem \ref{thm:Ruelle_vanishing_constant}.

Third, we compute the coefficient of the lead term in the expansion of $R^{2}(s;\chi)$ at $s=0$; see Theorem
\ref{thm:Ruelle_vanishing_constant}.  Indeed, this gives the coefficient of the lead term of $R(s;\chi)$ up to
sign.  For our purposes, this ambiguity of sign does not play a role since we next turn to the Fried conjecture
which itself involves $\vert R(0;\chi)\vert$, or more generally the absolute value of the coefficient of the lead term of
$R(s;\chi)$ at $s=0$.  Let us note that an
ambiguity of sign has been encountered elsewhere in the context of our study.
Specifically, in the introduction to \cite{Tu01}, the
author points out that R-torsion $\tau(M)$ has a ``well-known indeterminacy'' which includes a multiplicative
factor of $\pm 1$; see \cite[p.viii]{Tu01}.
As further stated in \cite{Tu01}, the sign determinacy led Tureav to define new
combinatorial structures on manifolds from which he defined other types of torsion; see \cite[chapter 18]{Tu01},
which evidentally goes beyond the topics under consideration in this article.

Fourth, we use our results and prove further instances of the Fried conjecture.  We consider two examples
where R-torsion has been computed.  In \cite{Ki94} and \cite{Ki96}, Kitano evaluates R-torsion
for a Seibert fibered space over a Riemann surface $M$ and an acyclic irreducible representation $\tilde{\rho}$
of $\pi_{1}(M)$ into $\mathrm{SL}(n,\mathbb{C})$.
From this data, we associate a naturally defined unitary multiplier system $\chi$ with non-zero weight and show that
$\tau(M;\tilde{\rho}) = \vert R(0;\chi)\vert ^{-1}$, thus proving the Fried conjecture; see Theorem \ref{thm: irreducible rep} below.
In \cite{Yamaguchi17}, Yamaguchi computes higher-dimensional R-torsion $\tau(M;\rho_{2N})$
for a Seifert fibered space and a certain acyclic representation $\rho_{2N}=\sigma_{2N} \circ \tilde{\rho}$ of dimension $2N$;
see \cite[Proposition 4.8]{Yamaguchi17}.  Associated to the data from \cite{Yamaguchi17}, we construct a weight one
multiplier system $\chi$ such tnat $\tau(M;\rho_{2N}) =\vert R(0;\chi)\vert$, thus extending the main result of \cite{Yamaguchi22}
which considers certain parametric values of $\rho_{2N}$.

Finally, we consider some explicit evaluations of the lead term in the Laurent expansion of $R(s;\chi)$ at $s=0$, in the case when $M$ is not compact.
In particular, we consider $M$ realized as the quotient of the upper half plane by the
congruence subgroup $\Gamma_{0}(N)$ and the multiplier system has weight zero and is associated to a
Dirichlet character modulo $N$.  As predicted by Fried, the absolute value of this lead term is given in terms of special values of arithmetic $L$-functions; see \cite[p. 162]{Fr86c}.

\subsection{Comparison to some recent results}

As stated, there are many fascinating results involving R-torsion, Ruelle zeta functions, the Fried
conjecture, and related fields.  Let us mention here just a few  recent results which
are related to our study.

In \cite{GP10} the authors studied the Ruelle and Selberg zeta functions for finite volume quotients
of hyperbolic manifolds without elliptic fixed points.  They proved that the Ruelle zeta function admits
a functional equation, and they evaluated the order of the divisor at $s=0$; see Theorem 1.1 and Theorem 1.2.
In this setting, the evaluation of the coefficient of the lead term in the Laurent expansion at $s=0$
was determined in \cite{Pa09}, and thus the Fried conjecture was proved.
These results were extended to non-unitary twists and compact surfaces by Spilioti in \cite{Sp18}
and \cite{Sp20} and Frahm and Spilioti in \cite{FS23}; see also
\cite{Go97}, \cite{GP08}, \cite{GP10}, \cite{Mul12}, \cite{Pa19} and \cite{CDDP22}
in the setting of (higher-dimensional) hyperbolic manifolds.

In \cite{Teo20} the author studies the Ruelle zeta function for any finite volume hyperbolic Riemann
surface in the case that the representation of $\pi_{1}(M)$  is trivial.  Teo proves a functional equation
of $R(s)$
and evaluates the divisor and leading coefficient at $s=0$; see Theorem 3.1 and Theorem 3.2 of \cite{Teo20}.
Though Teo considers cofinite hyperbolic Riemann surfaces with cusps and ramification points,
they only consider the case when $\chi$ is the trivial one-dimensional representation.

When the Riemann surface is negatively curved, but the curvature is not necessarily constant,
it is proved in \cite{DZ17} that its Ruelle zeta function vanishes at zero to the order given by the
absolute value of the Euler characteristic.  The surfaces considered in \cite{DZ17} are compact
and without elliptic fixed points.

Fried's conjecture for closed locally symmetric manifolds of nonpositive sectional curvature was proved in \cite{MS91}. In \cite{Sh18} the author proves the Fried conjecture for an arbitrary
acyclic and unitarily flat vector bundle on a closed locally symmetric reductive manifold such that the quotient
space is compact and odd dimensional.  Additional results towards a general Fried conjecture when $M$ is compact were recently proved in \cite{FS23}, \cite{Mu21} and \cite{Sh21}, see also \cite{DGRS20} for the first examples of geodesic flows in variable negative curvature where the Fried conjecture holds true. We refer an interested reader to the recent survey \cite{Shen21} (and the references therein) where the results on the Fried conjecture were described from the dynamical point of view.

In summary, there are a several recent publications related to the study of R-torsion and the Fried
conjecture for hyperbolic Riemann surfaces, and more generally quotients of symmetric spaces.
The presence of ramification points presents a significant technical difficulty in the study of the Ruelle
zeta function, its functional equation, and its lead asymptotic behavior near $s=0$.  The study undertaken in this article
focuses on these questions in the general setting of arbitrary finite volume hyperbolic Riemann surfaces, and arbitrary (admissible) weight unitary multiplier systems.

\subsection{Outline of the article}
In section \ref{sec: prelim} we establish notation and summarize existing results from the mathematical literature.
In particular, we recall results from \cite{Gong95} which express Selberg's zeta function as a regularized
determinant.
In section \ref{s:Rzeta} we obtain the functional equation of the Ruelle zeta function in Proposition
\ref{prop: funct eq R zeta} and then study the its divisor at $s=0$; see Theorem \ref{thm:Ruelle_vanishing_constant}.
In section \ref{sec: Kitano rep} we combine our results with those from \cite{Ki96} and prove further instances of the
Fried conjecture for R-torsion.
In section \ref{sec: Yam appl} we complete the results from \cite{Yamaguchi22} and
prove the Fried conjecture for the cases of higher R-torsion which were computed in \cite{Yamaguchi17}.
Finally, in section \ref{s:congruence}, we explicitly evaluate the lead coefficient of $R(s,\chi)$ for certain arithmetically
interesting examples.
As stated above, the computations in section \ref{s:congruence} confirm Fried's assertion from \cite[p. 162]{Fr86c} that the special value of the Ruelle zeta function at $s=0$ is
related to $L$-functions.

\section{Preliminaries}\label{sec: prelim}

\subsection{Basic notation}\label{sec:notation}

Let $\Gamma< {\rm SL}_2(\mathbb{R})$ be a discrete subgroup acting totally discontinuously on the upper half plane $\mathbb{H}=\{x+iy:\, x, y\in\mathbb{R},\, y>0\}$ such that $-I_2\in \Gamma$, where $I_2$ is the identity matrix.
The group $\Gamma$ can contain hyperbolic, elliptic  or parabolic elements.
Let $\mathcal{F}$ be a (Ford) fundamental domain for the action of $\Gamma$ on $\mathbb{H}$.
We further assume that the hyperbolic area of $\mathcal{F}$ is finite, meaning that $\Gamma$ is Fuchsian group of the first kind.
Let $\overline{\Gamma}$ be the projection of $\Gamma$ into $\mathrm{PSL}(2,\mathbb{R})$ and
$M$ be the Riemman surface associated to the quotient
space $\overline{\Gamma}\backslash\mathbb{H}$. Let $g$ be the genus of $\overline{\Gamma}\backslash\mathbb{H}$ \footnote{We will also say that $g$ is a genus of $\Gamma$} and assume it possesses $\tau$ and $\rho$ distinct parabolic cusps and
elliptic points, respectively. Let $\{\nu_{j}\}$, with $j=1,\cdots, \rho$, signify the orders of the
elliptic points.  With all this,
the hyperbolic area of  $\mathcal{F}$ is
\begin{equation}\label{eq: volume of F}
\omega(\mathcal{F}) = 2\pi\bigg(2g-2 +\sum_{j=1}^\rho \left(1-\frac{1}{\nu_j}\right)+\tau\bigg).
\end{equation}

If the number of cusps $\tau$  is positive, then for each cusp $\zeta\in\mathbb{R} \cup \{\infty\}$
its stabilizer is an infinite abelian group which is generated by $-I_2$ and a matrix $T$; the matrix $T$ is
conjugate  in $\mathrm{SL}(2,\mathbb{R})$ to the translation $U=\sm 1 & 1\\ 0 & 1\esm$.
Let us choose a complete system of representatives of $\Gamma$-equivalent classes of cusps of $\Gamma$ and matrices
$\{S_1,S_2,\dots,S_\tau \}$ which, together with $-I_2$, generate the stabilizers of the chosen cusps.
We assume that the fundamental domain $\mathcal{F}$
associated to this choice has nice geometric properties, as described in \cite[Remark 1.2.2]{Fischer}.
We say that $\{S_1,S_2,\dots,S_\tau \}$ are parabolic representatives for the cusps of $\Gamma$.

Let $\rho$ be the number of $\Gamma$-conjugacy classes of elliptic elements.
For each elliptic element $R\in \Gamma$, the centralizer ${\rm Z}(R)=\left\{\gamma\in \Gamma: \gamma R = R\gamma\right\}$ of $R$ in $\Gamma$ is a finite cyclic subgroup of order $2\nu$ for some integer $\nu\in \mathbb{Z}_{\geq 2}$.
The centralizer ${\rm Z}(R)$ is generated by an element $\widetilde{R}\in {\rm Z}(R)$, called primitive,
which is conjugate in $\mathrm{SL}(2,\mathbb{R})$  to the rotation
\begin{equation}\label{e:theta_nu_ellpitic}
\theta_{\frac{\pi}{\nu}}  = \bpm \cos\frac{\pi}{\nu} & -\sin\frac{\pi}{\nu} \\ \sin\frac{\pi}{\nu} & \cos\frac{\pi}{\nu}\ebpm.
\end{equation}
The number $\nu$ depends only on the $\Gamma$-conjugacy class of $R$.
Let $\{R_1, \cdots, R_{\rho}\}$ denote a complete set of primitive $\Gamma$-conjugacy classes of
elliptic elements such that each $R_j$ is conjugate in $\mathrm{SL}(2,\mathbb{R})$  to some $\theta_{\frac{\pi}{\nu_{j}}}$ for
all $j=1, \ldots, \rho$.

\subsection{Unitary multiplier systems}\label{sec: mult systems}

For a positive integer $m$ we fix a $m\times m$ unitary multiplier system $\chi$ of dimension $m$ and arbitrary admissible real weight $2k$.
By this we mean a  map from $\Gamma$ to the ring $\mathcal{U}(V)$ of unitary endomorphisms of an $m$-dimensional Hermitian vector
space $V$ which satisfies the following properties.  First,
\begin{equation}\label{eq: multsysprop1}
\chi(-I_2)=e^{-2\pi i k} I_V
\end{equation}
where $I_{V}$ is the identity element in $V$.  Second, for
any $\gamma$ and $\eta$ in $\Gamma$,  we have that
\begin{equation}\label{eq: multsysprop2}
\chi(\gamma\eta) =\sigma_{2k}(\gamma,\eta)\chi(\gamma)\chi(\eta),
\end{equation}
where $\sigma_{2k}$ is a weight $2k$ factor system; see \cite[Section 1.3]{Fischer}.
The weight $2k$ is \emph{admissible} if there exists a multiplier system with weight $2k$.

When $\tau=0$, meaning that $M$ is compact, then  the admissible values of $k$ are the real numbers
which lie in the set
\[A_m(\Gamma):=\frac{1}{m} \frac{2\pi}{\omega(\mathcal{F})}\frac{1}{\mathrm{l.c.m.}[\nu_1,\ldots,\nu_{\rho}]} \mathbb{Z};\]
see \cite[Proposition 1.3.6]{Fischer} or \cite[Proposition 2.3, p. 335]{Hejhal83}.
Note that in \cite{Hejhal83}  Hejhal uses $m$ to denote what we signify by $2k$.  If $\Gamma$ has no elliptic elements,
then $\rho=0$ and the term $\mathrm{l.c.m.}[\nu_1,\ldots,\nu_{\rho}]$ is set to equal $1$.
In view of \eqref{eq: volume of F}, for any $a\in \Z$ one has $\frac{a}{m}\in A_m(\Gamma)$, meaning $2\frac{a}{m}$ is an admissible weight for a unitary multiplier system of dimension $m$.

When $\tau\geq 1$, then from \cite[Proposition 2.2, p. 334]{Hejhal83} we have that all real values of $k$ are admissible.

\subsubsection{Action on parabolic representatives}\label{sss:parabolic}

 Assume that $\tau$, the number of cusps, is positive.
As above, let
$\{S_1,S_2,\dots,S_\tau \}$ denote a complete set of parabolic representatives for the cusps of $\Gamma$.
For each $j = 1\dots \tau,$ set
 $$
 V_j = \{ v \in V ~|~ \chi(S_j) v = v \}.
 $$
 Let $m_j = \dim(V_j)$, and define the \emph{parabolic degree of singularity} of $\chi$ to be
 $$
 \tau_0 = \sum\limits_{j=1}^{\tau} m_j.
 $$
 If $\tau_0 = 0$,  we say that $\chi$ is \emph{regular} otherwise $\chi$ we say that $\chi$ is \emph{singular};
 see \cite[Section~1.5 p.28]{Fischer}.  We also note that this use of the term \emph{singular} originates
 with Selberg; see \cite{Sel89} and references therein.

For each  $S_j \in \{S_1,S_2,\dots,S_\tau \},$  let $\eta_{j1},\ldots, \eta_{jm}$ be eigenvalues of $\chi(S_j)$ counted with multiplicity.
We can write
$$\eta_{jp} = e^{2\pi i \beta_{jp}}, $$
where we may assume, without loss of generality, that $\beta_{jp} = 0$ for $1 \leq p \leq m_j,$ and $\beta_{jp} \in (0,1) $ for $m_j < p \leq m$,
see \cite[Section~1.5 p.30]{Fischer}.  With this assumption, and
for any $j \in \{1,\dots,\tau \}$, we set
\[\beta_j= \sum_{p=1}^{m} \beta_{jp} = \sum_{p = m_j + 1}^m \beta_{jp}.\]
So $\beta_j=0$ when $m_j=m$.

\subsubsection{Action on primitive elliptic elements}\label{sss:elliptic}

There are $\rho$ distinct $\Gamma$-inconjugate elliptic subgroups of $\Gamma$.  If $\rho = 0$, then $M$ is smooth,
so for the remainder of this section let us assume $\rho \geq 1$.
For each $j=1,\ldots, \rho$ the multiplier system $\chi$ is such that
$$
\chi(R_j)^{\nu_j} =e^{-2\pi i k} I_V.
$$
Hence, $\chi(R_j)$ has $m$ eigenvalues which are necessarily of the form $\exp\left(-\frac{2\pi i}{\nu_j}(k+\alpha_{jp})\right)$,
where $\alpha_{jp}\in\{0,\ldots,\nu_j-1\}$ and $p=1,\ldots, m$; see \cite[p. 66]{Fischer}.

Let $E_j=\{v\in V: \chi(R_j)v=v\}$ and let $r_j=\mathrm{dim}(E_j)$. The \emph{elliptic degree of singularity} of $\chi$ is defined to be
$$
\tilde{\tau}_0:= \sum\limits_{j=1}^{\rho}r_j.
$$
If the value $k$ is not an integer, since no eigenvalue of $\chi_j(R)$ can be equal to one, we set $\tilde{\tau}_0=0$.

For a non-negative integer $\ell$, we further define quantities
$$
\alpha_j(\ell)=\sum_{p=1}^m \alpha_{jp}(\ell)
\,\,\,\,\,
\text{\rm and}
\,\,\,\,\, \tilde{\alpha}_j(\ell)=\sum_{p=1}^m \tilde{\alpha}_{jp}(\ell),
$$
where $\alpha_{jp}(\ell)$, respectively $\tilde{\alpha}_{jp}(\ell)$, is the residue of $\alpha_{jp}+\ell$, respectively
$-\alpha_{jp}+\ell$, modulo $\nu_j$ with $j=1,\ldots, \rho$.
For any $j\in\{1,\ldots,\rho\}$ and $\ell\in\{0,\ldots,\nu_j-1\}$, let $r_j(\ell)$ denote the number of values $\alpha_{jp}\in\{0,\ldots,\nu_j-1\}$ such that $\alpha_{jp}+\ell \equiv 0 \,(\mathrm{mod}\, \nu_j)$. Then $\sum\limits_{\ell=0}^{\nu_j-1} r_j(\ell) =m$ for each $j\in\{1,\ldots,\rho\}$.
We will extend $r_j(\ell)$ to all non-negative integers $\ell$ by imposing the condition that $r_{j}(\ell)$ as a function in $\ell$
is periodic with period $\nu_j$.  In particular, we have $r_j(\nu_j)=r_j(0)$.
Note that when $k\in \mathbb{Z}$, we then have that
\[r_j(k) = \#\left\{\alpha_{jp} :\, p\in \{1, \ldots, m\}, \, \alpha_{jp}+k\equiv 0\bmod{\nu_j}\right\}= \dim E_j = r_j.\]
So then
\begin{equation}\label{e:sum_rjk_int}
\sum_{j=1}^{\rho} r_j(k) = \sum_{j=1}^{\rho} r_j=\tilde{\tau}_0.
\end{equation}
Furthermore, for $j\in \{1, \ldots, \rho\}$, and by the definition of $r_j(\ell)$,  we observe that
\begin{equation}\label{eq: r(l) identity}
\frac{\alpha_j(\ell)-\alpha_j(\ell-1)}{\nu_j} = \frac{m}{\nu_j}-r_j(\ell).
\end{equation}

\begin{lemma}\label{lem: existence of mult system}
Let $\Gamma\subset \mathrm{SL}(2,\mathbb{R})$ be a Fuchsian group of the first kind of genus $g\geq 1 $ which does not contain parabolic elements and such that $-I_2\in\Gamma$.  Assume that $\Gamma$ contains $\rho\geq 1$ classes of elliptic elements,
with representatives $R_1,\ldots,R_\rho$ of orders $\nu_1,\ldots,\nu_\rho$ respectively.
Choose any positive integer  $m$ and integer $a\in\{1,\ldots,m\}$ such that $\gcd(a, m)=1$.  Then
there exists an $m\times m$ unitary multiplier system $\chi$ on $\Gamma$ with weight $2k=2\frac{a}{m}$.
Moreover, for any set of integers $\alpha_{jp}\in\{0,\ldots,\nu_j-1\}$, $p=1,\ldots,m$ and $j=1,\ldots,\rho$,
one can construct such a multiplier system $\chi$ so that eigenvalues of
$\chi(R_j)$ are $\lambda_{jp}:= \exp\left(-\frac{2\pi i}{\nu_j}(k+\alpha_{jp})\right)$ for all $p=1,\ldots, m$.
\end{lemma}

\begin{proof}
As discussed above, the weight $k=\frac{a}{m}$ is admissible.
Any unitary $m$-dimensional representation $\tilde{\chi}$ on $\Gamma$ with eigenvalues
for the action at $R_j$ equal to $\lambda_{jp}$ with $p=1,\ldots,m$ is such that
$\tilde{\chi}(R_j)^{\nu_j}=\exp(-2\pi i k) I_m$. This condition is precisely
what is needed for the action of any multiplier system on elliptic elements.
Recall that the group $\Gamma$ possesses the single relation that
$$
\Gamma \cong \langle A_i, B_i, R_j:\, \prod _{i=1}^{g} [A_i,B_i] \prod_{j=1}^{\rho}R_j =I_2,\, R_j^{\nu_j}=-I_2  \rangle,
$$
where $A_i, B_i$, $i=1,\ldots,g$  are representatives of hyperbolic classes and $[A_i, B_i]=A_iB_iA_i^{-1}B_i^{-1}$ is the commutator.
Therefore, we can define a mapping $\tilde{\chi}: \Gamma \to \mathrm{SL}(m,\mathbb{C})$ such that  $\tilde{\chi}(R_j)=\Lambda_j I_m$, where $\Lambda_j$ is the diagonal matrix with elements $\lambda_{jp}$, for $p=1,\ldots, m$, on the diagonal.
Furthermore, the action of $\tilde{\chi}$ on representatives $A_i, B_i$, $i=1,\ldots,g$ of hyperbolic classes of $\Gamma$
can be chosen that the transformation property \eqref{eq: multsysprop2} and the group law holds for some factor system
$\sigma_{2k}$.  With this, the proof of the lemma is complete.
\end{proof}

\subsection{The Gamma and related functions}
Throughout we will use basic properties of the Gamma function $\Gamma(s)$ and the Barnes $G$-function $G(s)$.
For the convenience of the reader, let us list here the properties will employ.

The Gamma function is defined for $\text{\rm Re}(s) > 0$ by
$$
\Gamma(s) = \int\limits_{0}^{\infty}e^{-t}t^{s}\frac{dt}{t}.
$$
It admits a meromorphic continuation to the whole complex plane with simple poles at zero and negative integers.
It is elementary to show that it
satisfies the functional equation $\Gamma(s+1)=s\Gamma(s)$ and the reflection formula, which is
$$
\Gamma(s)\Gamma(1-s) = \frac{\pi}{\sin(\pi s)}.
$$
For any integer $n\geq 1$, the Gamma function also satisfies the
multiplication formula
\begin{equation}\label{eq: gamma product f-la}
\Gamma\left(s\right) = (2\pi)^{\frac{1-n}{2}} n^{s-\frac{1}{2}} \prod_{\ell=0}^{n-1} \Gamma\left(\frac{s+\ell}{n}\right);
\end{equation}
see, for example,  \cite[Formula 8.335]{GR07}.

The Barnes
double Gamma function, also called the $G$-function, $G(s)$, is defined by the product expansion that
\begin{equation*}\label{e:barnesG}
G(s+1) = (2\pi)^{\frac{s}{2}} e^{-\frac{s(s+1)}{2}-\frac{1}{2}c s^2} \prod_{n=1}^\infty \bigg\{\bigg(1+\frac{s}{n}\bigg)^n e^{-s+\frac{s^2}{2n}}\bigg\}
\end{equation*}
satisifes $G(s+1) = \Gamma(s)G(s)$; see \cite[\S5.17]{dlmf}). Note that the above formulas continue meromorphically to $s\in \mathbb{C}$
except the the poles of Gamma functions and Barnes's $G$-functions.

Both the Gamma function and the Barnes $G$-function appear as factors in the functional equations for the
Selberg zeta function and the Ruelle zeta function.  As stated above, we will derive a functional equation
for the Ruelle zeta function which involves only $\sin(s)$ and not $\Gamma(s)$ or $G(s)$.  In doing so, we will
use the following lemma; for completeness, we will give a proof.

\begin{lemma}\label{lem: product of sines}
Let $\nu >1$ be an integer. Then for any real number for which $x\notin\mathbb{Z}$ one has that
\begin{equation}\label{eq: prod sines x}
\prod_{\ell=1}^{\nu} \sin\left(\frac{\ell-x}{\nu}\pi\right)= 2^{1-\nu} \sin(\pi x).
\end{equation}
If $x=n\in\{1,\ldots,\nu\}$, then
\begin{equation}\label{eq: prod sines int k}
\prod\limits_{\ell\in\{1,\ldots,\nu\}\setminus\{n\}} \sin\left(\frac{\ell-n}{\nu}\pi\right)=(-1)^{n-1} \nu 2^{1-\nu}.
\end{equation}
\end{lemma}

\begin{proof}
To prove \eqref{eq: prod sines x} we apply the reflection formula for the Gamma function, followed by the product formula \eqref{eq: gamma product f-la}.
Specifically, we have that
\begin{align*}
\prod_{\ell=1}^\nu \sin\left(\frac{\ell-x}{\nu}\pi\right)
& = \pi^{\nu} \prod_{\ell=1}^\nu \bigg(\Gamma\left(\frac{\ell-x}{\nu}\right)\Gamma\left(1-\frac{\ell-x}{\nu}\right)\bigg)^{-1}
\\ & = \pi^\nu \bigg(\prod_{\ell=0}^{\nu-1}\Gamma\left(\frac{1-x+\ell}{\nu}\right)\bigg)^{-1}
\bigg(\prod_{\ell=0}^{\nu-1} \Gamma\left(\frac{x+\ell}{\nu}\right)\bigg)^{-1}
\\ & = \pi^{\nu} \frac{(2\pi)^{1-\nu}}{\Gamma\left(1-x\right)\Gamma\left(x\right)}
= 2^{1-\nu} \sin(\pi x).
\end{align*}
To prove \eqref{eq: prod sines int k}, we set $x=n+y$ in \eqref{eq: prod sines x}.
Then
\[\sin\left(-\frac{y}{\nu}\pi \right) \prod_{\ell\in \{1, \ldots, \nu\}\setminus\{n\}} \sin\left(\frac{\ell-n-y}{\nu}\pi\right)
= 2^{1-\nu}\sin(\pi(n+y)) = 2^{1-\nu}(-1)^n \sin(\pi y).\]
Dividing by $\sin(-\frac{y}{\nu}\pi)$ on both sides and then taking $y\to \infty$, we get \eqref{eq: prod sines int k}.
\end{proof}


\subsection{Scattering determinant}\label{ss:scattering}

Assume that $\tau\geq 1$ and that the degree of singularity $\tau_0$ of a fixed multiplier system $\chi$ of weight $2k$ is positive.
Let $\varphi(s;\chi)$ denote the determinant of the $\tau_0\times\tau_0$ automorphic scattering matrix $\Phi(s;\chi)$; see \cite[Section 1.5.8]{Fischer}.
The function $\varphi(s;\chi)$ is a meromorphic function of order at most two.
Furthermore, $\varphi(s;\chi)$ is holomorphic for $\Re(s) > \frac{1}{2}$, except for a finite number of poles,
and it satisfies the functional equation
\begin{equation}\label{functeq phi}
\varphi(s;\chi)\varphi(1-s;\chi)=1.
\end{equation}
For $\Re(s)> 1$ we have that
\begin{equation} \label{phiDirich}
\varphi (s;\chi)=\left( \frac{ \Gamma(s) \,\Gamma \left( s-\frac{1}{2}
\right) }{\Gamma \left( s -k\right) \Gamma \left( s +k\right) }\right) ^{\tau_0}\overset{\infty }{\underset
{n=1}{\sum }}\frac{d(n) }{g_{n}^{2s}},
\end{equation}
where $0< g_{1} < g_{2}< ...$, $d(n) \in \mathbb{C}$ with $d(1)\neq 0$ and the series is absolutely convergent; see \cite[formula (1.5.3.)]{Fischer}.
We will rewrite \eqref{phiDirich} in a slightly different form.
Let $c_{1}=-2\log{g_{1}}$, and set $u_{n}=g_{n}/g_{1}>1$.

Then for $\text{Re}(s) > 1$ we can write
\begin{equation}\label{e:varphi_Ltildevarphi}
\varphi( s;\chi) =L(s)\tilde{\varphi}(s;\chi)
\end{equation}
where
\begin{equation} \label{eqPhiA}
L(s) =\left( \frac{ \Gamma(s) \,\Gamma \left( s-\frac{1}{2}
\right) }{\Gamma \left( s -k\right) \Gamma \left( s +k\right) }\right) ^{\tau_0} d(1) e^{c_{1}s}
\end{equation}
and
\begin{equation} \label{eqPhiB}
\tilde{\varphi}(s;\chi) =1+\overset{\infty }{\underset{n=2}{\sum }}\frac{a(n)}{u_{n}^{2s}},
\end{equation}
with $a(n) \in \mathbb{C}$.
The series \eqref{eqPhiB} converges absolutely for $\Re(s)>1$.

Let $n_0=n_0(\Gamma;\chi)\in\mathbb{Z}$ denote the exponent in the Taylor, or Laurent, series expansion of $\tilde{\varphi}(s;\chi)$
as $s=0$ and let $a_{n_0}$ denote the lead term in this expansion.
In other words, let us write
\begin{equation}\label{eq: vartilde at zero}
s^{-n_0} \tilde{\varphi}(s; \chi) = a_{n_0} + O(s)\quad \text{ as } s\to 0.
\end{equation}

The scattering determinant $\varphi(s;\chi)$ is real-valued for $s\in \mathbb{R}$; see \cite[formula (1.5.7.)]{Fischer}.
It follows that $d(1) \in \mathbb{R}$.
Therefore, $L(s)$, defined by \eqref{eqPhiA} is real-valued for real $s$, and then $\tilde{\varphi}(s;\chi)$ is also real-valued for real $s$.
Consequently, we have that
$a_{n_0} \in\mathbb{R}$.

When $\tau_0=1$, the scattering determinant $\varphi (s;\chi)$ has a simple pole at $s=1$.  Then, from the functional equation \eqref{functeq phi},
we have that $\varphi (s;\chi)$ vanishes to order one at $s=0$. In this case, we can say more about $n_0$. Namely, we have that
\begin{equation}\label{eq: n0 when tau0=1}
n_0=\left\{
      \begin{array}{ll}
        0, & \text{  if  } k=0; \\
        1, & \text{  if  } k\in\mathbb{Z}\setminus \{0\}; \\
        2, & \text{  if  } k \notin\mathbb{Z}.
      \end{array}
    \right.
\end{equation}
When $\tau_0>1$, for generic Fuchsian groups, the behavior of the scattering determinant at the point $s=0$ is not yet fully understood.

\subsection{The Selberg zeta function and Ruelle zeta function}\label{sec:Selberg_and_Ruelle}

Following \cite[p.496]{Hejhal83}, the Selberg zeta function is defined for $\Re(s)>1$ by the absolutely convergent product
\begin{equation}\label{e:Z_def}
Z(s; \chi) = \prod_{\substack{P_0, \\ {\rm tr}(P_0)>2}} \prod_{\ell=0}^\infty \det\left(I_m - \chi(P_0) N(P_0)^{-s-\ell}\right);
\end{equation}
the product runs through all primitive hyperbolic elements $P_0$ of the group $\Gamma$, and $N(P_0)$ is the norm of the element $P_0$.
As above, $m$ is the dimension of the multiplier system $\chi$.
For $\mathrm{Re}(s)>1$, the Ruelle zeta function is defined by the absolutely convergent product
\begin{equation}\label{e:R_def}
R(s; \chi) = \prod_{\substack{P_0, \\ {\rm tr}(P_0)>2}}\mathrm{det}(I_m- \chi(P_0) N(P_0)^{-s} ).
\end{equation}
From the absolute convergence of \eqref{e:Z_def} and \eqref{e:R_def}, it is immediate that
\begin{equation}\label{eq: ruelle in terms of Selb zeta}
R(s; \chi) = \frac{Z(s; \chi)}{Z(s+1; \chi)},
\end{equation}
for $\Re(s)>1$. The Selberg zeta function is studied in great detail in \cite{Hejhal83}.
It is shown that the Selberg zeta function, hence the Ruelle zeta function, admits a meromorphic continuation to all $s \in \mathbb{C}$.
Furthermore, the Selberg zeta function admits a functional equation which relates the value at $s$ to the value at $1-s$.  One of
the first points in our analysis is to re-write the functional equation of $Z(s;\chi)$ in such a way that we derive
a functional equation for the Ruelle zeta function which relates the value at $s$ to the value at $-s$ with an explicit factor which is a meromorphic function of order one.  These considerations
follow from the work from \cite{Gong95} which expresses the Selberg zeta function
as a regularized product associated to a Laplacian operator.

{\it We assume that the weight $2k$, dimension $m$ unitary multiplier system $\chi$ is fixed and, for the sake of brevity, will
omit $\chi$ from the notation in the zeta functions and the scattering determinant.}

\subsection{Zeta regularized determinants}\label{sec:regularizedZeta}
Consider the weight $2k$ Maass-Laplacian
\begin{equation*}
  \Delta_{2k} = -y^2\left(\frac{\partial^2}{\partial x^2} + \frac{\partial^2}{\partial y^2} \right) +2kiy \frac{\partial}{\partial x}
\end{equation*}
which acts on the space of twice continuously differentiable functions $f:\mathbb{H}\to V$ that satisfy the transformation property
\begin{equation}\label{eq: transf prop functions}
f(\gamma z) = \exp(2ik \arg(cz +d)) \chi(\gamma) f(z) \quad \text{ for all }
 \gamma = \bpm * & * \\ c & d\ebpm \in \Gamma.
\end{equation}
Here we choose $\arg(cz+d) \in (-\pi, \pi]$.
We identify $\Delta_{2k}$ with its self-adjoint extension to the Hermitian space of all $L^2$ functions on $\mathcal{F}$ satisfying the
transformation property \eqref{eq: transf prop functions}.
When $\Gamma$ has no cusps, so $\tau=0$, or when $\tau\geq 1$ but the degree of singularity $\tau_0$ of $\chi$ is zero, then
the operator $\Delta_{2k}$ has only the discrete spectrum with eigenvalues $\lambda_0 \leq \lambda_1\leq \cdots \leq \lambda_j = s_j(1-s_j) \leq  \cdots$ tending to $+\infty$.
Otherwise, $\Delta_{2k}$ has both the discrete spectrum
and the continuous spectrum $[1/4,\infty)$ with the spectral measure expressed through the logarithmic derivative of the scattering determinant.

Following \cite[p.441-442]{Gong95}, we define the spectral zeta function $\zeta(w, s)$ for $\Re(s)>1$ and $\Re(w)$ sufficiently large by
\begin{equation}\label{e:spectralzeta_def}
\zeta(w, s) = \sum_{j\geq 0} (s_j(1-s_j) - s(1-s))^{-w} -\frac{1}{4\pi} \int\limits_{-\infty}^\infty \frac{\varphi'}{\varphi}\left(\frac{1}{2}+ir\right) \frac{dr}{((\frac{1}{4}+r^2)-s(1-s))^{w}},
\end{equation}
where, as above, $\varphi$ the scattering matrix. 
As proved in \cite{Gong95}, one can take $\Re(s)$ and $\Re(w)$ large enough so that
the series and the integral in \eqref{e:spectralzeta_def} converge absolutely.
Going further, it is proved in \cite{Gong95} that for $\Re(s)>1$ and $|k|-s\notin\mathbb{Z}_{\geq 0}$, the spectral zeta function $\zeta(w,s)$ possesses meromorphic continuation to the whole $w-$plane and is holomorphic at $w=0$.  With this result,
the zeta regularized determinant $\det\left(\Delta_{2k}-s(1-s)\right)$ for $\Re(s)>1$ and $|k|-s\notin\mathbb{Z}_{\geq 0}$ is defined as
\begin{equation}\label{e:det_def}
\det\left(\Delta_{2k}-s(1-s)\right) = e^{-\frac{\partial}{\partial w}\left.\zeta(w, s)\right\vert_{w=0}}.
\end{equation}
With this, \cite[Theorem 3]{Gong95} proves a relation between the zeta regularized product \eqref{e:det_def}
and the Selberg zeta function $Z(s)=Z(s;\chi)$.  Indeed, for $\Re(s)>1$ and $|k|-s\notin\mathbb{Z}_{\geq 0}$ one has that
\begin{equation}\label{e:det_Delta}
\det\left(\Delta_{2k}-s(1-s)\right)
= Z(s) Z_{I}(s)  Z_{\rm ell}(s) Z_{{\rm par}, 1}(s) e^{-bs(1-s)+\tilde{c}}
\end{equation}
for certain, explicitly computed, functions $ Z_{I}(s),\, Z_{\rm ell}(s)$ and $Z_{{\rm par}, 1}(s)$, which are given below.  Also,
in \cite[Section 5]{Gong95} it is proved that that $b=0$ and
$\tilde{c}$ is computed.  For our purposes, the constant $\tilde{c}$ does not play a role, so we refer to
\cite{Gong95} for the evaluation of $\tilde{c}$.

For $s\in \mathbb{C}\setminus(-\infty, |k|]$, it is shown that
\begin{multline}\label{e:ZI_def}
Z_{I}(s) = \exp\bigg\{\frac{m\omega(\mathcal{F})}{2\pi} \bigg(s\log(2\pi) + s(1-s) + \left(\frac{1}{2}+k\right) \log\Gamma\left(s+k\right)
\\ + \left(\frac{1}{2}-k\right) \log \Gamma\left(s-k\right) - \log G\left(s+k+1\right)-\log G\left(s-k+1\right)\bigg)\bigg\},
\end{multline}
\begin{multline}\label{e:Zell_def}
Z_{\rm ell}(s) = \prod_{j=1}^\rho \bigg\{\nu_j^{m(1-\frac{1}{\nu_j})s} \left(\Gamma\left(s-k\right)\Gamma\left(s+k\right)\right)^{-\frac{1}{2}m (1-\frac{1}{\nu_j})}
\\ \times \prod_{\ell=0}^{\nu_j-1} \bigg(\Gamma\left(\frac{s-k+\ell}{\nu_j}\right)^{\frac{\alpha_j(\ell)}{\nu_j} } \Gamma\left(\frac{s+k+\ell}{\nu_j}\right)^{\frac{\tilde{\alpha}_j(\ell)}{\nu_j} }\bigg)\bigg\},
\end{multline}
and
\begin{multline}\label{e:Zpar_def}
Z_{{\rm par}, 1}(s)
= 2^{-m\tau s}
  \prod_{j=1}^\tau \bigg\{\bigg(\frac{\Gamma\left(s+k\right)}{\Gamma\left(s-k\right)}\bigg)^{\frac{m}{2}-\beta_j}
 \prod_{p=m_j+1}^{m} (\sin(\pi \beta_{jp}))^{-s} \bigg\}
\\\times \left(s-\frac{1}{2}\right)^{\frac{1}{2}{\rm tr}(I_{\tau_0} -\Phi(\frac{1}{2}))} \bigg({\Gamma\left(s-k\right)}{\Gamma\left(s\right)\Gamma\left(s+\frac{1}{2}\right)}\bigg)^{\tau_0}.
\end{multline}

Let us emphasize that in all formulas above, and in the sequel below, we use the principal branch of the logarithm with its argument in
the range $(-\pi , \pi]$ in order to define rational powers of functions.

\subsection{Representations on Seifert fibered spaces} In this subsection we assume that $\tau=0$, so then
$M=\overline{\Gamma}\backslash \mathbb{H}$ is compact and assume that the genus $g$ of $M$ greater than zero.
Let $X=\overline{\Gamma}\backslash \mathrm{PSL}(2,\mathbb{R})$. Then $X$ is a Seifert fibered $3$-manifold which
admits a right action by the
circle group $K\subset \mathrm{PSL}(2,\mathbb{R})$ so that $X/K=M$;  \cite[Section 2]{Fr86c}.  It is proved
in \cite{Fr86c} that for any positive $m$ and admissible weight $2k$ every $m\times m$ multiplier system $\chi$
on $\Gamma$ induces a unitary representation on the fundamental group $\pi_1 (X)$ such that the
R-torsion associated to this  $m$-dimensional
unitary representation  equals $|R(0;\chi)|^{-1}$.
In the terminology from \cite[p. 149]{Fr86c}, such
representations are called \emph{classical}.   We refer to \cite[Section 2]{Fr86c} for
a further discussion of classical representations and their role in the study of R-torsion.

\emph{In this paper we consider representations of a Seifert fibered space that are not necessarily unitary,
but are acyclic which is a well-known necessary condition for the R-torsion to be well defined.}

Assume now that $X$ is a orientable Seifert fibered space given by the
Seifert index $\{b,(o,g); (\nu_1,\beta_1),\ldots,(\nu_\rho,\beta_\rho)\} $ with $g\geq 1$;
see \cite{JN83}.
The fundamental group of $X$ admits a presentation given by
\begin{equation}\label{eq: prez of fund group}
 \pi_1 (X)
=\left<a_1,b_1,\ldots,a_g,b_g,q_1,\ldots,q_\rho, h \bigg|
{ [a_i,h]=[b_i,h]=[q_j,h]=1, \atop q_j^{\nu_j}h^{\beta_j}=1, \, q_1\cdots q_\rho[a_1,b_1]\cdots [a_g,b_g] =h^b }\right>.
\end{equation}

\subsubsection{Irreducible acyclic $\mathrm{SL}(m,\mathbb{C})$-representations}\label{subs: irreducible kitano}

Let $\tilde{\rho}:\pi_1(X) \to \mathrm{SL}(m,\mathbb{C})$ be an irreducible acyclic representation.
The irreducibility of $\tilde{\rho}$ implies the existence of an $m$-th root of unity $\lambda$ such that $\tilde{\rho}(h)=\lambda I_m$.
Specifically, there exists $a\in \{1, \ldots, m-1\}$ with $\gcd(a, m)=1$ such that $\lambda=e^{2\pi i \frac{a}{m}}$.
If $\tilde{\rho}(h)= I_m$, the representation is not acyclic; see \cite[Proposition 4.1]{Ki96}.
In other words, in order for
an irreducible representation $\tilde{\rho}$ to be acyclic one must have $\lambda\neq 1$ (i.e. $a\neq 0$).

For each $j=1, \ldots, \rho$, let $\tilde{s}_j$ denote the homotopy class of the core of the exceptional fibers $S_j$ with the index $(\nu_j,\beta_j)$.
Each $\tilde{s}_j$ can be expressed as $q_j^{\mu_j}h^{\alpha_j}$ for $\mu_j, \alpha_j$ such that $\alpha_j\nu_j - \beta_j\mu_j=-1$ and $0<\mu_j<\nu_j$.
Then, when employing the fact that  $q_j^{\nu_j}h^{\beta_j}=1$, we get that
\[\tilde{\rho}(\tilde{s}_j)^{\nu_j} = \tilde{\rho}(q_j)^{\mu_j\nu_j}\tilde{\rho}(h)^{\nu_j\alpha_j}= \tilde{\rho}(h)^{\nu_j\alpha_j - \beta_j\mu_j}
= \exp\bigg(-2\pi i \frac{a}{m}\bigg)I_m.\]
In conclusion, the eigenvalues of $\tilde{\rho}(\tilde{s}_j)$ are numbers of the form
\[\lambda_{jp}:= \exp\left(\frac{-2\pi i}{\nu_j}(k+\alpha_{jp})\right)\]
where $k=\frac{a}{m}$ and $\alpha_{jp}$, for $p=1,\ldots, m$ are integers which come from the set $\{0,\ldots,\nu_j-1\}$.
By Lemma \ref{lem: existence of mult system}, one can associate to this representation $\tilde{\rho}$ an $m\times m$
unitary multiplier system $\chi_{\tilde{\rho}}$ of weight $2k$ such that the eigenvalues of $\chi_{\tilde{\rho}}(R_j)$
coincide with the eigenvalues of $\tilde{\rho}(\tilde{s}_j)$ for all $j=1,\ldots,\rho$.

\subsubsection{Even dimensional acyclic representations} \label{sebsect: yam represent}
Let $\tilde{\rho}$ be a $\mathrm{SL}(2,\mathbb{C})$-representation of the fundamental group $\pi_1(X)$ such that $\tilde{\rho}(h)=-I_2$.
Then, it is proved in \cite[Lemma 4.2]{Yamaguchi17} that for all $j=1,\ldots,\rho$ the
eigenvalues of $\tilde{\rho}(\tilde{s}_j)$ are given by $\exp(\pm i\pi \eta_j/\nu_j)$ for some odd integers $\eta_j$ with
$j=1,\ldots,\rho$.  As above, $\tilde{s}_j$ denotes the homotopy class of the core of the exceptional fibers $S_j$.

For an even positive integer $2N$, let $\tilde{\sigma}_{2N}$ denote an irreducible $2N$-dimensional representation
of $\mathrm{SL}(2,\mathbb{R})$ and define the representation $\rho_{2N}=\tilde{\sigma}_{2N} \circ \tilde{\rho}$.
It is further deduced in \cite{Yamaguchi17} that if $\tilde{\rho}(h)=-I_2$, then the representation $\rho_{2N}$ is acyclic.
Moreover, according to \cite[Remark 3.2]{Yamaguchi22}, if the eigenvalues of $\tilde{\rho}(\tilde{s}_j)$
are given by $\exp(\pm i\pi \eta_j/\nu_j)$, for some odd integers $\eta_j$, $j=1,\ldots,\rho$, then the eigenvalues
of $\rho_{2N}(\tilde{s}_j)$ are given by $\exp(\pm \pi i\eta_j/\nu_j), \exp(\pm 3\pi i \eta_j/\nu_j),\ldots, \exp(\pm (2N-1) \pi i\eta_j/\nu_j)$.

The representation $\rho_{2N}$ is an acyclic but not necessarily irreducible representation of a Seifert fibered space $X$
with  Seifert index $\{2-2g,(o,g); (\nu_1,\nu_1-1),\ldots,(\nu_\rho,\nu_\rho-1)\} $. The
associated R-torsion was computed in \cite{Yamaguchi17}.  The analogue of the Fried conjecture in the case $\eta_{j}=1$
was proved in \cite{Yamaguchi22}.   We will establish the Fried conjecture for general $\eta_{j}$ in Section \ref{sec: Yam appl}.

\section{Ruelle zeta function}\label{s:Rzeta}

\subsection{The functional equation of the Selberg zeta function}

In \cite{Hejhal83} and \cite{Fischer}, the authors study in detail the Selberg zeta function,
and its meromorphic continuation and functional equation were established; see, for example, equations
(5.7) on page 499 through (5.10) on page 501 of \cite{Hejhal83} or pages 115--116 of \cite{Fischer}.
For our purposes, it is necessary to find an expression for the factor in the functional equation which is somewhat
more tractable, which is the purpose of this section.
Throughout  the multiplier system $\chi:\Gamma \to \mathcal{U}(V)$ is fixed, so we will suppress it from the notation
for $Z(s;\chi)$ and simply write  $Z(s)=Z(s;\chi)$.

\begin{proposition} \label{prop: funct eq sel zeta}
The Selberg zeta function $Z(s)=Z(s;\chi)$ admits a
meromorphic continuation to the whole complex plane and satisfies the functional equation
\begin{equation}\label{e:Z-fe}
Z(1-s) = \kappa(s) Z(s)
\end{equation}
with
\begin{equation}\label{e:Selberg_factor}
\kappa(s) = \frac{Z_{I}(s)  Z_{\rm ell}(s) Z_{{\rm par}, 1}(s)}{Z_{I}(1-s)Z_{\rm ell}(1-s)   Z_{{\rm par}, 1}(1-s)}\varphi(s)
\end{equation}
and where the terms in \eqref{e:Selberg_factor} are as follows.
\begin{enumerate}[label=(\roman*)]
\item
The function $\varphi(s) = \varphi(s; \chi)$ is the determinant of the scattering matrix,
as described in Section \ref{ss:scattering}.
\item
For $s\in \mathbb{C}\setminus \left((-\infty, |k|]\cup [1-|k|, +\infty)\right)$, the factor $Z_{I}(s)/Z_{I}(1-s)$, which is
associated to the identity element $I \in \Gamma$, is
\begin{align}\label{e:ZI_fe}\notag
\frac{Z_{I}(s)}{Z_{I}(1-s)}
&= \exp\bigg\{\frac{m\omega(\mathcal{F})}{2\pi} \bigg((-1+2s) \log(2\pi)
+ k \log \frac{\sin(\pi(s-k))}{\sin(\pi(s+k))}
\\& - \frac{1}{2}\log\frac{\Gamma\left(s+k\right)}{\Gamma\left(1-s+k\right)}\frac{\Gamma\left(s-k\right)}{\Gamma\left(1-s-k\right)}- \log  \frac{G(s+k)G(s-k)}{G(1-s+k)G(1-s-k)}
\bigg)\bigg\}.
\end{align}
\item
For $s\in \mathbb{C}\setminus \left((-\infty, |k|]\cup [1-|k|, +\infty)\right)$, the factor $Z_{\rm ell}(s)/Z_{\rm ell}(1-s)$, which is associated to the elliptic elements in $\Gamma$, as described in Section \ref{sss:elliptic}, is
\begin{align}\label{e:Zell_fe}\notag
\frac{Z_{\rm ell}(s)}{Z_{\rm ell}(1-s)}
&=  2^{m\sum_{j=1}^{\rho} (1-\frac{1}{\nu_j})} \sin\left(\pi (s+k)\right)^{-m\sum_{j=1}^{\rho} \frac{1}{\nu_j}}
\bigg(\frac{\sin(\pi(s-k))}{\sin(\pi(s+k))}\bigg)^{\frac{1}{2}m \sum_{j=1}^{\rho}(1-\frac{1}{\nu_j})}
\\ & \times
\frac{\prod_{j=1}^\rho\prod_{\ell=1}^{\nu_j} \sin\left(\pi \left(\frac{-s-k+\ell}{\nu_j}\right)\right)^{\frac{\alpha_j(\ell)}{\nu_j}+r_j(\ell) } }
{\prod_{j=1}^\rho\prod_{\ell=0}^{\nu_j-1} \sin\left(\pi \left(\frac{s-k+\ell}{\nu_j}\right)\right)^{\frac{\alpha_j(\ell)}{\nu_j} } }.
\end{align}
\item
 For $s\in \mathbb{C}\setminus \left((-\infty, |k|]\cup [1-|k|, +\infty)\right)$,  the factor $Z_{\mathrm{par},1}(s)/Z_{\mathrm{par},1}(1-s)$, which is associated to the parabolic elements in $\Gamma$, as described in Section \ref{sss:parabolic}, is
\begin{align}\label{e:Zpar_fe}\notag
\frac{Z_{{\rm par}, 1}(s)}{Z_{{\rm par}, 1}(1-s)} &= 2^{-m\tau (2s-1)}
\bigg(\frac{\sin(\pi(s-k))}{\sin(\pi(s+k))}\bigg)^{\frac{m}{2}\tau -\sum_{j=1}^{\tau} \beta_j}
\prod_{j=1}^\tau \prod_{p=m_j+1}^{m} (\sin(\pi \beta_{jp}))^{1-2s}
\\& \times (-1)^{\frac{1}{2}{\rm tr}(I_{\tau_0} -\Phi(\frac{1}{2}))}
\bigg(\frac{\Gamma\left(s-k\right)}{\Gamma\left(1-s-k\right)}
\frac{\Gamma\left(1-s\right)\Gamma\left(\frac{3}{2}-s\right)}{\Gamma\left(s\right)\Gamma\left(s+\frac{1}{2}\right)}\bigg)^{\tau_0}.
\end{align}
\end{enumerate}
\end{proposition}

\begin{remark}
By convention, the empty product is taken to be equal to $1$, meaning that when $\rho=0$ the elliptic factor equals one and
in case when $\tau=0$, the parabolic factor and the scattering determinant $\varphi$ are equal to one.  When $\tau >0$,
but the multiplier system is regular, meaning $\tau_{0}=0$, then the last factor in \eqref{e:Zpar_fe} is identically one.
\end{remark}

\begin{proof}
The proof follows from the discussion in Sections \ref{sec:Selberg_and_Ruelle} and \ref{sec:regularizedZeta}.
To begin, one can show the existence of a function $D(s(1-s))$ which is symmetric with respect to the change
in variables $s \mapsto 1-s$
such that
\begin{equation} \label{eq: delta in terms of D}
\det\left(\Delta_{2k}-s(1-s)\right)
= D(s(1-s)) \varphi(s)^{-\frac{1}{2}}.
\end{equation}
In the case when $\chi$ is trivial, \eqref{eq: delta in terms of D} is stated on \cite[p.68]{Teo20}.
In the general setting we are considering here, one begins with the
spectral side of the trace formula, as stated in \cite[Theorem 3]{Gong95}, from
which the proof of \cite[Proposition 2.5]{Teo20} applies with only straightforward notational changes.
As such, one then has, as in \cite{Teo20}, that
\begin{equation*}
\det\left(\Delta_{2k}-(1-s)s\right)
=  Z(1-s) Z_{I}(1-s) Z_{\rm ell}(1-s)  Z_{{\rm par}, 1}(1-s) e^{-bs(1-s)+\tilde{c}},
\end{equation*}
 for the constants $b$ and $\tilde{c}$ as in \eqref{e:det_Delta}.
When combined with \eqref{eq: delta in terms of D} and \eqref{e:det_Delta}, one gets that
\begin{equation}\label{eq:Zeta_formula}
 Z(1-s) Z_{I}(1-s)  Z_{\rm ell}(1-s) Z_{{\rm par}, 1}(1-s)\varphi(1-s)^{\frac{1}{2}}= Z(s) Z_{I}(s) Z_{\rm ell}(s)  Z_{{\rm par}, 1}(s)\varphi(s)^{\frac{1}{2}}.
\end{equation}
Immediately,  the functional equation \eqref{e:Z-fe} follows from \eqref{eq:Zeta_formula} and the functional equation for the scattering determinant,
which is stated in \eqref{functeq phi}.

Let us now use the formulas from Sections \ref{sec:Selberg_and_Ruelle} and \ref{sec:regularizedZeta} in order to derive the
expressions given in \eqref{e:ZI_fe}, \eqref{e:Zell_fe} and \eqref{e:Zpar_fe} above.
 Note that
\begin{equation*}
\left(\mathbb{C}\setminus (-\infty, |k|]\right) \cap \left(\mathbb{C}\setminus [1-|k|, +\infty)\right) = \mathbb{C}\setminus \left((-\infty, |k|]\cup [1-|k|, +\infty)\right).
\end{equation*}
For $s\in \mathbb{C}\setminus \left((-\infty, |k|]\cup [1-|k|, +\infty)\right)$,
the definition \eqref{e:ZI_def} of the identity contribution gives that
\begin{align}\label{eq: id contrib beginning}\notag
\frac{Z_{I}(s)}{Z_{I}(1-s)}
&= \exp\bigg\{\frac{m\omega(\mathcal{F})}{2\pi} \bigg((-1+2s) \log(2\pi)
+ \frac{1}{2}\log\frac{\Gamma\left(s+k\right)}{\Gamma\left(1-s+k\right)}\frac{\Gamma\left(s-k\right)}{\Gamma\left(1-s-k\right)}
\\& + k \log\frac{\Gamma\left(s+k\right)}{\Gamma\left(1-s+k\right)}\frac{\Gamma\left(1-s-k\right)}{\Gamma\left(s-k\right)}- \log \frac{G\left(s+k+1\right)}{G\left(1-s+k+1\right)}\frac{G\left(s-k+1\right)}{G\left(1-s-k+1\right)}\bigg)\bigg\}.
\end{align}
The Barnes $G$-function satisfies the relation that
$G(s+1) = \Gamma(s) G(s)$; see, for example, \cite[\S5.17]{dlmf}.
Therefore
\begin{equation*}\label{eq: id contr G term}
\frac{G(s+k+1)G(s-k+1)}{G(1-s+k+1)G(1-s-k+1)}
 = \frac{\Gamma\left(s+k\right)\Gamma\left(s-k\right) G(s+k)G(s-k)}{\Gamma\left(-s+k+1\right)\Gamma\left(-s-k+1\right) G(-s+k+1)G(-s-k+1)}.
\end{equation*}
Substituting this expression into \eqref{eq: id contrib beginning} gives that
\begin{align}\label{eq: id contrib beginning 2}\notag
\frac{Z_{I}(s)}{Z_{I}(1-s)}
= &\exp\bigg\{\frac{m\omega(\mathcal{F})}{2\pi} \bigg((-1+2s) \log(2\pi)
- \frac{1}{2}\log\frac{\Gamma\left(s+k\right)}{\Gamma\left(1-s+k\right)}\frac{\Gamma\left(s-k\right)}{\Gamma\left(1-s-k\right)}
\\& + k \log\frac{\Gamma\left(s+k\right)}{\Gamma\left(1-s+k\right)}\frac{\Gamma\left(1-s-k\right)}{\Gamma\left(s-k\right)}
- \log  \frac{G(s+k)G(s-k)}{G(-s+k+1)G(-s-k+1)}
\bigg)\bigg\}.
\end{align}
The reflection formula for the Gamma function implies that
\begin{equation*}
\frac{\Gamma\left(s\pm k\right)}{\Gamma\left(1-s\pm k\right)}
= \Gamma\left(s+k\right)\Gamma\left(s-k\right)\frac{\sin(\pi(s\mp k))}{\pi}
\end{equation*}
which together with \eqref{eq: id contrib beginning 2} yields \eqref{e:ZI_fe}.

In order to deduce  the formula \eqref{e:Zell_fe} for the elliptic contribution,
we start with \eqref{e:Zell_def} and rearrange the product with powers involving $\tilde{\alpha}(\ell)$ as
\begin{equation}\label{eq: prod gammas with tilde}
\prod_{\ell=0}^{\nu_j-1} \Gamma\left(\frac{s+k+\ell}{\nu_j}\right)^{\frac{\tilde{\alpha}_j(\ell)}{\nu_j}}
= \prod_{\ell=0}^{\nu_j-1} \Gamma\left(\frac{s+k-1-\ell}{\nu_j}+1\right)^{\frac{\tilde{\alpha}_j(\nu_j-1-\ell)}{\nu_j}}.
\end{equation}
Recall that
\begin{equation*}
\alpha_j(\ell)=\sum_{p=1}^m \alpha_{jp}(\ell)\quad \text{ and } \quad
\tilde{\alpha}_j(\ell)=\sum_{p=1}^m \tilde{\alpha}_{jp}(\ell).
\end{equation*}
Furthermore, $\alpha_{jp}(\ell)$ (resp. $\tilde{\alpha}_{jp}(\ell)$) is the residue of $\alpha_{jp}+\ell$
(resp. $-\alpha_{jp}+\ell$) modulo $\nu_j$.
Hence
\begin{equation*}
\tilde{\alpha}_j(\nu_j-1-\ell) = \sum_{p=1}^m \tilde{\alpha}_{jp}(\nu_j-1-\ell)
\end{equation*}
and $\tilde{\alpha}_{jp}(\nu_j-1-\ell)\in \{0,\ldots, \nu_j-1\}$ is such that
\begin{equation*}
\tilde{\alpha}_{jp}(\nu_j-1-\ell)\equiv -\alpha_{jp} +\nu_j-1-\ell
\equiv ( -\alpha_{jp}(\ell) +\nu_j-1)\bmod{\nu_j}.
\end{equation*}
Since $\alpha_{jp}(\ell)\in \{0,\ldots, \nu_j-1\}$, then  $-\alpha_{jp}(\ell)+\nu_j-1 \in \{0,\ldots, \nu_j-1\}$, so we get
\begin{equation*}
\tilde{\alpha}_{jp}(\nu_j-1-\ell) = -\alpha_{jp}(\ell)+\nu_j-1
\end{equation*}
and then
\begin{equation*}
\tilde{\alpha}_j(\nu_j-1-\ell) = \sum_{p=1}^m \tilde{\alpha}_{jp}(\nu_j-1-\ell)=m(\nu_j-1) -\alpha_j(\ell).
\end{equation*}
Thus \eqref{eq: prod gammas with tilde} becomes that
\begin{equation*}\label{eq: prod gammas with tilde 2}
\prod_{\ell=0}^{\nu_j-1} \Gamma\left(\frac{s+k+\ell}{\nu_j}\right)^{\frac{\tilde{\alpha}_j(\ell)}{\nu_j}}
= \prod_{\ell=0}^{\nu_j-1} \Gamma\left(\frac{s+k-1-\ell}{\nu_j}+1\right)^{m(1-\frac{1}{\nu_j}) -\frac{\alpha_j(\ell)}{\nu_j}}.
\end{equation*}
Then for $s\in \mathbb{C}\setminus (-\infty, |k|]$, we arrive that
\begin{align*}
Z_{\rm ell}(s) = \prod_{j=1}^\rho \bigg\{&\nu_j^{m(1-\frac{1}{\nu_j})s}
\left(\Gamma\left(s-k\right)\Gamma\left(s+k\right)\right)^{-\frac{1}{2}m (1-\frac{1}{\nu_j})}
\\& \times \prod_{\ell=0}^{\nu_j-1} \bigg(\Gamma\left(\frac{s-k+\ell}{\nu_j}\right)\bigg)^{\frac{\alpha_j(\ell)}{\nu_j} }
\prod_{\ell=0}^{\nu_j-1} \Gamma\left(\frac{s+k-1-\ell}{\nu_j}+1\right)^{m(1-\frac{1}{\nu_j}) -\frac{\alpha_j(\ell)}{\nu_j}}
\bigg\}.
\end{align*}
We can get a similar formula for $Z_{\rm ell}(1-s)$ for $s\in \mathbb{C}\setminus [1-|k|, +\infty)$.
When
combining these formulas and applying the reflection formula for $\Gamma(s)$, we get
for $s\in \mathbb{C}\setminus \left((-\infty, |k|]\cup [1-|k|, +\infty)\right)$ that
\begin{align*}
\frac{Z_{\rm ell}(s)}{Z_{\rm ell}(1-s)}
= &\prod_{j=1}^\rho \bigg\{\nu_j^{m(1-\frac{1}{\nu_j})(-1+2s)}
\bigg(\frac{\Gamma\left(s-k\right)\Gamma\left(s+k\right)}{\Gamma\left(1-s-k\right)\Gamma\left(1-s+k\right)}\bigg)^{-\frac{1}{2}m (1-\frac{1}{\nu_j})}
\\& \times \prod_{\ell=0}^{\nu_j-1} \bigg(\sin\left(\pi \left(\frac{s-k+\ell}{\nu_j}\right)\right)\bigg)^{-\frac{\alpha_j(\ell)}{\nu_j} }
\prod_{\ell=0}^{\nu_j-1} \bigg(\sin\left(\pi \left(\frac{1-s-k+\ell}{\nu_j}\right)\right)\bigg)^{\frac{\alpha_j(\ell)}{\nu_j}}
\\& \times \prod_{\ell=0}^{\nu_j-1} \Gamma\left(\frac{-1+s+k-\ell}{\nu_j}+1\right)^{m(1-\frac{1}{\nu_j})}
\prod_{\ell=0}^{\nu_j-1}\Gamma\left(\frac{-s+k-\ell}{\nu_j}+1\right)^{-m(1-\frac{1}{\nu_j})}
\bigg\}.
\end{align*}
Recall the identities \eqref{eq: r(l) identity} and \eqref{eq: prod sines x}.  By rearranging the products, we get
$j\in \{1, \ldots, \rho\}$ that
\begin{align*}
\frac{Z_{\rm ell}(s)}{Z_{\rm ell}(1-s)}
&= \prod_{j=1}^\rho \bigg\{\nu_j^{m(1-\frac{1}{\nu_j})(-1+2s)}
2^{m(1-\frac{1}{\nu_j})} \sin\left(\pi (s+k)\right)^{-\frac{m}{\nu_j}}
\\& \times
\bigg(\frac{\Gamma\left(s-k\right)\Gamma\left(s+k\right)}{\Gamma\left(1-s-k\right)\Gamma\left(1-s+k\right)}\bigg)^{-\frac{1}{2}m (1-\frac{1}{\nu_j})}
\\& \times
\frac{\prod_{\ell=1}^{\nu_j} \sin\left(\pi \left(\frac{-s-k+\ell}{\nu_j}\right)\right)^{\frac{\alpha_j(\ell)}{\nu_j}+r_j(\ell) } }
{\prod_{\ell=0}^{\nu_j-1} \sin\left(\pi \left(\frac{s-k+\ell}{\nu_j}\right)\right)^{\frac{\alpha_j(\ell)}{\nu_j} } }
\frac{\prod_{\ell=0}^{\nu_j-1} \Gamma\left(\frac{s+k+\nu_j-1-\ell}{\nu_j}\right)^{m(1-\frac{1}{\nu_j})}}
{\prod_{\ell=0}^{\nu_j-1}\Gamma\left(\frac{1-s+k+\nu_j-1-\ell}{\nu_j}\right)^{m(1-\frac{1}{\nu_j})}}
\bigg\}.
\end{align*}
Here we also used that $\alpha_j(\nu_j) = \alpha_j(0)$.
After changing $\nu_j-1-\ell$ to $\ell$, and then applying \eqref{eq: gamma product f-la}, we get
\begin{equation*}
\frac{\prod_{\ell=0}^{\nu_j-1}\Gamma\left(\frac{s+k+\nu_j-1-\ell}{\nu_j}\right)}{\prod_{\ell=0}^{\nu_j-1}\Gamma\left(\frac{1-s+k+\nu_j-1-\ell}{\nu_j}\right)}
= \frac{\prod_{\ell=0}^{\nu_j-1}\Gamma\left(\frac{s+k+\ell}{\nu_j}\right)}{\prod_{\ell=0}^{\nu_j-1}\Gamma\left(\frac{1-s+k+\ell}{\nu_j}\right)}
= \frac{\nu_j^{-2s+1} \Gamma\left(s+k\right)}{\Gamma\left(1-s+k\right)}
\end{equation*}
Hence, we have
\begin{multline*}
\frac{Z_{\rm ell}(s)}{Z_{\rm ell}(1-s)}
= \prod_{j=1}^\rho \bigg\{
2^{m(1-\frac{1}{\nu_j})} \sin\left(\pi (s+k)\right)^{-\frac{m}{\nu_j}}
\bigg(\frac{\Gamma\left(1-s-k\right)\Gamma\left(s+k\right)}{\Gamma\left(s-k\right)\Gamma\left(1-s+k\right)}\bigg)^{\frac{1}{2}m (1-\frac{1}{\nu_j})}
\\ \times
\frac{\prod_{\ell=1}^{\nu_j} \sin\left(\pi \left(\frac{-s-k+\ell}{\nu_j}\right)\right)^{\frac{\alpha_j(\ell)}{\nu_j}+r_j(\ell) } }
{\prod_{\ell=0}^{\nu_j-1} \sin\left(\pi \left(\frac{s-k+\ell}{\nu_j}\right)\right)^{\frac{\alpha_j(\ell)}{\nu_j} } }
\bigg\}.
\end{multline*}
When now using  the reflection formula for Gamma functions, we arrive at \eqref{e:Zell_fe}.

It remains to prove \eqref{e:Zpar_fe}.
Recall \eqref{e:Zpar_def}.  For $s\in \mathbb{C}\setminus \left((-\infty, |k|]\cup [1-|k|, +\infty)\right)$, we have that
\begin{align}\notag
\frac{Z_{{\rm par}, 1}(s)}{Z_{{\rm par}, 1}(1-s)}
= &2^{-m\tau (2s-1)}
\prod_{j=1}^\tau \bigg\{\bigg(\frac{\Gamma\left(s+k\right)}{\Gamma\left(s-k\right)} \frac{\Gamma\left(1-s-k\right)}{\Gamma\left(1-s+k\right)}\bigg)^{\frac{m}{2}-\beta_j}
\prod_{p=m_j+1}^{m} (\sin(\pi \beta_{jp}))^{1-2s} \bigg\}
\\& \times (-1)^{\frac{1}{2}{\rm tr}(I_{\tau_0} -\Phi(\frac{1}{2}))}
\bigg(\frac{\Gamma\left(s-k\right)}{\Gamma\left(1-s-k\right)}
\frac{\Gamma\left(1-s\right)\Gamma\left(\frac{3}{2}-s\right)}{\Gamma\left(s\right)\Gamma\left(s+\frac{1}{2}\right)}\bigg)^{\tau_0}.
\end{align}
By using the reflection formula for the ratio of Gamma functions in the first line,
the proof of  \eqref{e:Zpar_fe} follows immediately.
\end{proof}

\subsection{The functional equation for the Ruelle zeta function}\label{s:RZ_fe}

Let us recall the equation \eqref{eq: ruelle in terms of Selb zeta}
which implies that the meromorphic continuation of $R(s; \chi)$ is a consequence
of the meromorphic continuation of $Z(s; \chi)$.  We will now use the functional equation for $Z(s; \chi)$, as given in Proposition \ref{prop: funct eq sel zeta}, to derive a simple functional equation for $R(s; \chi)$. Again, since the multiplier system $\chi$ is fixed,
we will omit using $\chi$ in the notation whenever possible.

\begin{proposition} \label{prop: funct eq R zeta}
With the notation as above, the Ruelle zeta function satisfies the functional equation
\begin{equation}\label{eq: ruella zeta funct eq}
R(s)\varphi(s)= \left(R(-s)\varphi(-s)\right)^{-1}H(s)
\end{equation}
where the meromorphic function $H(s)$ is defined for $s\in \mathbb{C}$ by
\begin{align}\label{e:H}\notag
H(s)
= & 2^{2m(2g-2)}
\big(\sin(\pi(s+k))\sin(\pi(-s+k))\big)^{m(2g-2-\rho+\tau)}
\\& \notag \times \prod_{j=1}^{\rho} \prod_{\ell=1}^{\nu_j} \bigg(\sin\left(\pi \left(\frac{-s-k+\ell}{\nu_j}\right)\right)\sin \left(\pi \left(\frac{s-k+\ell}{\nu_j}\right)\right)\bigg)^{-r_j(\ell)}
\\& \times
\bigg(\prod_{j=1}^\tau \prod_{p=m_j+1}^{m} (\sin(\pi \beta_{jp}))\bigg)^{-2}
\bigg(\frac{(-s-k)(s-k)}{s(-s)(s+\frac{1}{2})(-s+\frac{1}{2})}\bigg)^{\tau_0}.
\end{align}

\end{proposition}

\begin{proof}
By \eqref{e:Z-fe} we have that
\begin{equation} \label{eq: R funct equations 2}
R(s) = \frac{Z(s)}{Z(s+1)} = \frac{Z(s)\kappa(s+1)}{Z(-s)}
\,\,\,\,\,
\text{\rm and}
\,\,\,\,\,
R(-s) = \frac{Z(-s)}{Z(1-s)} = \frac{Z(-s)}{Z(s)\kappa(s)}.
\end{equation}
Then
\begin{equation*}
R(s)R(-s) = \frac{\kappa(s+1)}{\kappa(s)}
=  H(s) (\varphi(s)\varphi(-s))^{-1}.
\end{equation*}
where
\begin{equation*}
H(s) = \frac{Z_{I}(1+s)  Z_{I}(1-s)}{Z_{I}(s)Z_{I}(-s)}
\frac{Z_{\rm ell}(1+s)  Z_{\rm ell}(1-s)}{Z_{\rm ell}(s)Z_{\rm ell}(-s)}
\frac{Z_{{\rm par}, 1} (1+s)  Z_{{\rm par}, 1} (1-s)}{Z_{{\rm par}, 1}(s) Z_{{\rm par}, 1}(-s)}.
\end{equation*}
By Lemma \ref{lem: factors of funct eq Ruelle}, we get \eqref{e:H}.
Note that we have used the formula for $\omega(\mathcal{F})$ in \eqref{eq: volume of F}.
Also, the exponents appearing in $H(s)$ in \eqref{e:H} are all integers.
As a result, the expression in \eqref{e:H}, which initially is defined for $s\in \mathbb{C}\setminus(-\infty, \infty)$,
can be extended to an order one meromorphic function on the whole complex plane.
\end{proof}

Let us now state a second form of the functional equation for the Ruelle zeta function.

\begin{corollary}\label{cor:RZ_fe}
With the notation as above, specifically \eqref{phiDirich}, \eqref{e:varphi_Ltildevarphi} and \eqref{eqPhiA}, the Ruelle zeta function satisfies the functional equation
\begin{equation}\label{eq: ruella zeta funct eq2}
R(s)\tilde{\varphi}(s)= c(\chi,\Gamma)^{-2} \left(R(-s)\tilde{\varphi}(-s)\right)^{-1}H_{1}(s)
\end{equation}
where the meromorphic function $H_1(s)$ is defined for $s\in \mathbb{C}$ as
\begin{align}\label{eq: ruella zeta funct eq2_factor}
H_1(s)
&= \big(\sin(\pi(s+k))\sin(\pi(-s+k))\big)^{m(2g-2-\rho+\tau)-\tau_0}
\\& \notag\times \prod_{j=1}^{\rho} \prod_{\ell=1}^{\nu_j} \bigg(\sin\left(\pi \left(\frac{-s-k+\ell}{\nu_j}\right)\right)\sin \left(\pi \left(\frac{s-k+\ell}{\nu_j}\right)\right)\bigg)^{-r_j(\ell)} \left(\frac{\sin(\pi s) \cos(\pi s)}{s}\right)^{\tau_0}\notag
\end{align}
and
\begin{equation}\label{e:c-RZfe2}
c(\chi,\Gamma)=  2^{-m(2g-2)} d(1) \prod_{j=1}^\tau \prod_{p=m_j+1}^{m} (\sin(\pi \beta_{jp})).
\end{equation}
\end{corollary}

\begin{proof}
Recall that $\varphi(s)= L(s) \tilde{\varphi}(s)$, where $L(s)$ is given by \eqref{eqPhiA}.
Then, with the notation from Proposition \ref{prop: funct eq R zeta}, we can write \eqref{eq: ruella zeta funct eq} as
\[R(s)\tilde\varphi(s) = \left(R(-s)\tilde\varphi(-s)\right)^{-1}H(s) (L(-s) L(s))^{-1}\]
where
\[\left(L(s)L(-s)\right)^{-1}
= d(1)^{-2} \left(\frac{\Gamma\left(s-\frac{1}{2}\right) \Gamma\left(-s-\frac{1}{2}\right)\Gamma\left(s\right)\Gamma\left(-s\right)}
{\Gamma(s-k)\Gamma(-s-k)\Gamma(s+k)\Gamma(-s+k)}\right)^{-\tau_0}.\]
Since $\Gamma(s)\Gamma(-s)=-\pi/(s\sin(\pi s))$, we get that
\[\left(L(s)L(-s)\right)^{-1}\bigg(\frac{(s+k)(s-k)}{s^2(s-\frac{1}{2})(s+\frac{1}{2})}\bigg)^{\tau_0}
= d(1)^{-2} \bigg(\frac{\sin(\pi s)\cos(\pi s)}{\sin(\pi(s-k)) \sin(\pi(s+k))}
\frac{1}{s}\bigg)^{\tau_0}.\]
The claimed functional equation \eqref{eq: ruella zeta funct eq2} then follows from the formula for $H(s)$ given in \eqref{e:H}.

\end{proof}

The following lemma completes the proof of Proposition \ref{prop: funct eq R zeta}.

\begin{lemma} \label{lem: factors of funct eq Ruelle}
With the notation as above, for $s\in \mathbb{C}\setminus (-\infty, \infty)$, we have the following identities.
\begin{enumerate}[label=(\roman*)]
\item For the function $Z_{I}(s)$, we have that
\begin{equation} \label{eq: ide term ruelle}
\frac{Z_I(1+s)}{Z_I(-s)} \frac{Z_I(1-s)}{Z_I(s)}
= \bigg(4\sin(\pi(s+k))\sin(\pi(-s+k))\bigg)^{\frac{m\omega(\mathcal{F})}{2\pi}}.
\end{equation}

\item  For the function $Z_{\rm ell}(s)$, we have that
\begin{multline}\label{eq: ell term ruelle}
\frac{Z_{\rm ell}(1-s)}{Z_{\rm ell}(s)} \frac{Z_{\rm ell}(1+s)} {Z_{\rm ell}(-s)}
= 2^{-2m\sum_{j=1}^{\rho} (1-\frac{1}{\nu_j})} \big(\sin(\pi(s+k)) \sin(\pi(-s+k))\big)^{\sum_{j=1}^{\rho} \frac{m}{\nu_j}}
\\ \times \prod_{j=1}^{\rho} \prod_{\ell=1}^{\nu_j} \bigg(\sin\left(\pi \left(\frac{-s-k+\ell}{\nu_j}\right)\right)
\sin \left(\pi \left(\frac{s-k+\ell}{\nu_j}\right)\right)\bigg)^{-r_j(\ell)}.
\end{multline}

\item  For the function $Z_{{\rm par},1}(s)$, we have that
\begin{equation}\label{eq: par term ruelle}
\frac{Z_{{\rm par}, 1}(1-s)} {Z_{{\rm par}, 1}(s)}\frac{Z_{{\rm par}, 1}(1+s)} {Z_{{\rm par}, 1}(-s)}
= 2^{-2m\tau}
\bigg(\prod_{j=1}^\tau \prod_{p=m_j+1}^{m} (\sin(\pi \beta_{jp}))\bigg)^{-2}
\bigg(\frac{(-s-k)(s-k)}{s(-s)(s+\frac{1}{2})(-s+\frac{1}{2})}\bigg)^{\tau_0}.
\end{equation}
\end{enumerate}
\end{lemma}

\begin{proof}
The proof involves direct, albeit involved, manipulations of the special functions $\Gamma(s)$ and $G(s)$.
We start with the formulas from Proposition \ref{prop: funct eq sel zeta}.

First we consider the identity contribution.  Recall from \eqref{e:ZI_fe} that for $s\in \mathbb{C}\setminus (-\infty, \infty)$
we can write
\begin{align*}
&\frac{Z_{I}(1+s)Z_I(1-s)}{Z_I(s) Z_I(-s)}
\\ &= \exp\bigg\{\frac{m\omega(\mathcal{F})}{2\pi} \bigg(2\log(2\pi)
- \frac{1}{2}\log\frac{\Gamma\left(1-s+k\right)}{\Gamma\left(s+k\right)}\frac{\Gamma\left(1-s-k\right)}{\Gamma\left(s-k\right)}\frac{\Gamma\left(1+s+k\right)}{\Gamma\left(-s+k\right)}\frac{\Gamma\left(1+s-k\right)}{\Gamma\left(-s-k\right)}
\\ & \hskip .15in - \log  \frac{G(1-s+k)G(1-s-k)}{G(s+k)G(s-k)}\frac{G(1+s+k)G(1+s-k)}{G(-s+k)G(-s-k)}
\bigg)\bigg\}
\end{align*}
Recall that $G(s+1) = \Gamma(s) G(s)$ and that $\Gamma(s)\Gamma(1-s) = \pi/\sin(\pi s)$.  Then
\begin{align*}
\frac{Z_{I}(1+s)Z_I(1-s)}{Z_I(s) Z_I(-s)}
&= \exp\bigg\{\frac{m\omega(\mathcal{F})}{2\pi} \bigg(2\log 2+\log\big(\sin(\pi(-s+k))\sin(\pi(s+k))\big)\bigg)\bigg\}
\\& = (4\sin(\pi(-s+k))\sin(\pi(s+k)))^{\frac{m\omega(\mathcal{F})}{2\pi}},
\end{align*}
which proves \eqref{eq: ide term ruelle}.

For the elliptic contribution.  It is straightforward to begin with
\eqref{e:Zell_fe} and then simplifying it to get \eqref{eq: ell term ruelle} assuming
that $s\in \mathbb{C}\setminus (-\infty, \infty)$.  For the sake of space, we will omit
these computations.

It remains to analyze the parabolic contribution.
For $s\in \mathbb{C}\setminus (-\infty, \infty)$, we begin with \eqref{e:Zpar_fe} to get that
\begin{align*}
\frac{Z_{{\rm par}, 1}(1-s)} {Z_{{\rm par}, 1}(s)}
&\frac{Z_{{\rm par}, 1}(1+s)} {Z_{{\rm par}, 1}(-s)}
= 2^{-2m\tau}
\prod_{j=1}^\tau \prod_{p=m_j+1}^{m} (\sin(\pi \beta_{jp}))^{-2}
\\& \times
\bigg(\frac{\Gamma\left(s-k\right)}{\Gamma\left(1-s-k\right)}\frac{\Gamma\left(-s-k\right)}{\Gamma\left(1+s-k\right)}
\frac{\Gamma\left(1-s\right)\Gamma\left(\frac{3}{2}-s\right)}{\Gamma\left(s\right)\Gamma\left(s+\frac{1}{2}\right)}
\frac{\Gamma\left(1+s\right)\Gamma\left(\frac{3}{2}+s\right)}{\Gamma\left(-s\right)\Gamma\left(-s+\frac{1}{2}\right)}\bigg)^{-\tau_0}.
\end{align*}
By repeatedly using the identity $\Gamma(s+1)=s\Gamma(s)$, expression \eqref{eq: par term ruelle} follows.
\end{proof}

\subsection{The order of the divisor of $R(s;\chi)$ and the lead term at $s=0$}

In this section  we will compute ${\rm ord}_R(0)$, the order of the divisor of $R(s; \chi)$ at $s=0$.  Also, we will evaluate the
constant term $s^{-{\rm ord}_R(0)} R(s; \chi)$ as $s\to 0$, up to sign.
As we will show below, ${\rm ord}_R(0)$ depends on the topology of the space $M=\overline{\Gamma}\backslash \mathbb{H}$,
the degree of singularity of $\chi$ on the generators of the parabolic and elliptic subgroups of $\Gamma$, and
the order of the divisor of the Dirichlet series part $\tilde{\varphi}(s;\chi)$ of the scattering matrix $\varphi(s;\chi)$ at $s=0$.
Throughout we will rely on the notation established in section \ref{sec: mult systems}.  Since
the multiplier $\chi$ will be fixed, we will omit $\chi$ from the notation.

The main results of this section are summarized in the following theorem.

\begin{theorem}\label{thm:Ruelle_vanishing_constant}
With notation as above, we have that
\begin{equation*}\label{e:Ruelle_order}
{\rm ord}_R(0) = -n_0+\frac{1}{2}{\rm ord}_{H_1}(0)
\end{equation*}
where
\begin{equation*}\label{e:H1_order}
{\rm ord}_{H_1}(0) = \begin{cases} 0 & \text{ when } k\notin \mathbb{Z}, \\
2\big(m(2g-2+\rho+\tau)-\tau_0-\tilde{\tau}_0\big) & \text{ when } k\in \mathbb{Z}. \end{cases}
\end{equation*}
Furthermore, when $k\notin\mathbb{Z}$, we have that
\begin{align}\label{e:Ruelle_limit_square_knotint}\notag
\lim_{s\to 0} \big(s^{2n_0} R(s)^2\big)
&= a_{n_0}^{-2} c(\chi, \Gamma)^{-2}
\\& \times \pi^{\tau_0}
(\sin(\pi k))^{2m(2g-2-\rho+\tau)-2\tau_0}
\prod_{j=1}^{\rho} \prod_{\ell=1}^{\nu_j} \sin\left(\pi \left(\frac{-k+\ell}{\nu_j}\right)\right)^{-2r_j(\ell)}.
\end{align}
In the case when $k\in \mathbb{Z}$, we have that
\begin{align}\label{e:Ruelle_limit_square_kint}\notag
\lim_{s\to 0} \big(s^{-2{\rm ord}_R(0)}R(s)^2\big)
&= a_{n_0}^{-2} c(\chi, \Gamma)^{-2}
\\& \times  \pi^{{\rm ord}_{H_1}(0)+\tau_0}
\prod_{j=1}^{\rho}\nu_j^{2r_j(k_j)}
\prod_{j=1}^{\rho} \prod_{\substack{\ell\in \{1, \ldots, \nu_j\}\\ \ell\neq k_j}} \sin\left(\pi \left(\frac{-k+\ell}{\nu_j}\right)\right)^{-2r_j(\ell)},
\end{align}
where $n_0 = {\rm ord}_{\tilde{\varphi}}(0)$, $a_{n_0} = \lim_{s\to 0} (s^{-n_0}\tilde{\varphi}(s))$ and $c(\chi,\Gamma)$ is given by \eqref{e:c-RZfe2}.
In the last product of \eqref{e:Ruelle_limit_square_kint}, $k_j$ is an integer between $1$ and $\nu_j$ such that $k_j\equiv k\bmod{\nu_j}$.
\end{theorem}

\begin{remark}
Recall that in Section \ref{ss:scattering} it is shown that
$d(1)$ and $a_{n_0}$ are real numbers.  As a result, the right hand sides of \eqref{e:Ruelle_limit_square_knotint} and \eqref{e:Ruelle_limit_square_kint}
are positive real numbers.  Therefore, we have determined the values of the
limits
\begin{equation}\label{eq:un_squared_limits}
\lim_{s\to 0} \big(s^{n_0} R(s)\big)
\,\,\,\,\,
\text{\rm and}
\,\,\,\,\,
\lim_{s\to 0} \big(s^{-{\rm ord}_R(0)}R(s)\big)
\end{equation}
uniquely up to a factor of $\pm 1$.
A similar sign ambiguity also arises in \cite{Fr86c} where the author considers the case when $M$ is compact
and $\chi$ is a unitary representation.  In our analysis, the challenge in determining the sign of the limits
\eqref{eq:un_squared_limits} occurs when
imposing necessary restrictions to ensure that all branches of the logarithm are principal.  It may be possible
to overcome the obstacles for some cases when $k\in\mathbb{Z}$; we carry out such
computations in Section \ref{sss:leadatsequal0} for $k=0$.  However, as stated in the introduction, the
value of R-torsion itself is defined up to sign, and ultimately we will use Theorem \ref{thm:Ruelle_vanishing_constant} to
obtain further instances of the Fried conjecture; see Sections \ref{sec: Kitano rep} and \ref{sec: Yam appl} below.
As such, the results in \eqref{e:Ruelle_limit_square_knotint} and \eqref{e:Ruelle_limit_square_kint} suffice for
our purposes.
\end{remark}

\begin{remark}\label{rem: order of van}\rm
There are various interesting examples of Theorem \ref{thm:Ruelle_vanishing_constant}.
For example, if $\Gamma$ is co-compact and the weight is an even integer, then
$\ord_R(0)$ depends only upon the topology of the surface $M=\overline{\Gamma}\backslash \mathbb{H}$
as well as on the elliptic degree of
singularity of $\chi$.

If $g\geq 1$, and $\chi$ acts non-trivially on elliptic elements, then $s=0$ is a zero
of $R(s;\chi)$ of order $m(2g-2) + m\rho-\tilde{\tau}_0>0$.
However, when $\Gamma$ is co-compact and the weight is not an even integer, then the
order of the divisor of $R(s;\chi)$ at $s=0$ is zero, meaning that $R(s;\chi)$ is non-vanishing at $s=0$.

If $\Gamma$ is non-compact and the degree $\tau_0$ of singularity of $\chi$ with weight $2k\notin 2\mathbb{Z}$
equals $1$ then, from \eqref{eq: n0 when tau0=1}, we have that $\ord_R(0)=-n_0 = -2$.

When combining these examples, one sees that the order of the divisor of $R(s;\chi)$ at $s=0$ can be positive, negative, or zero.
\end{remark}

\subsubsection{Proof of Theorem \ref{thm:Ruelle_vanishing_constant}}


We begin with the
functional equation \eqref{eq: ruella zeta funct eq2} in Corollary \ref{cor:RZ_fe}, which states that
\begin{equation*}
R(s)R(-s) \tilde{\varphi}(s)\tilde{\varphi}(-s) = c(\chi,\Gamma)^{-2} H_{1}(s)
\end{equation*}
where $H_1(s)$ is given in \eqref{eq: ruella zeta funct eq2_factor} and $c(\chi, \Gamma)$ in \eqref{e:c-RZfe2}.
From \eqref{eq: vartilde at zero}, one has that
$$
\ord_{\tilde{\varphi}}(0)=n_0.
$$
Therefore we have $2{\rm ord}_R(0) + 2n_0 = {\rm ord}_{H_1}(0)$ and
\begin{equation}\label{e:ordR^2_ordH1}
(-1)^{{\rm ord}_R(0)+n_0} \bigg(\lim_{s\to 0}\big(s^{-{\rm ord}_R(0)-{\rm ord}_{\tilde{\varphi}}(0)}R(s)\tilde{\varphi}(s)\big)\bigg)^2
= c(\chi, \Gamma)^{-2} \lim_{s\to 0} s^{-{\rm ord}_{H_1}(0)} H_1(s).
\end{equation}
Recall from \eqref{eq: ruella zeta funct eq2_factor} that for $s\in \mathbb{C}\setminus(-\infty, \infty)$,
$H_{1}(s)$ is given as
\begin{align*}
H_1(s)
&=
\big(\sin(\pi(s+k))\sin(\pi(-s+k))\big)^{m(2g-2-\rho+\tau)-\tau_0}
\\ &\times \prod_{j=1}^{\rho}
\prod_{\ell=1}^{\nu_j} \bigg(\sin\left(\pi \left(\frac{-s-k+\ell}{\nu_j}\right)\right)
\sin \left(\pi \left(\frac{s-k+\ell}{\nu_j}\right)\right)\bigg)^{-r_j(\ell)}
\left(\frac{\sin(\pi s) \cos(\pi s)}{s}\right)^{\tau_0}.
\end{align*}

Consider the case when $k\not\in \mathbb{Z}$.  It is immediate that
$H_1(s)$ is not vanishing as $s\to 0$, meaning that
${\rm ord}_{H_1}(0)=0$, hence ${\rm ord}_R(0) = -n_0$.
Moreover, we have that
\[
\lim_{s\to 0} H_1(s)
= 2^{-2m\sum_{j=1}^{\rho}(1-\frac{1}{\nu_j})} \pi^{\tau_0}
(\sin(\pi k))^{2m(2g-2-\rho+\tau)-2\tau_0}
\prod_{j=1}^{\rho} \prod_{\ell=1}^{\nu_j} \sin\left(\pi \left(\frac{-k+\ell}{\nu_j}\right)\right)^{-2r_j(\ell)},
\]
which evidentially is a positive real number.
We have
\begin{equation*}
\lim_{s\to 0} \big(s^{2n_0} R(s)^2\big)
= a_{n_0}^{-2} c(\chi, \Gamma)^{-2} \lim_{s\to 0} H_1(s),
\end{equation*}
which proves \eqref{e:Ruelle_limit_square_knotint}.

Let us now assume that $k\in \mathbb{Z}$.
For $s\in \mathbb{C}\setminus (-\infty, \infty)$, we have that
\begin{align*}
H_1(s)
&=(-\sin^2(\pi s))^{m(2g-2-\rho+\tau)-\tau_0} \prod_{j=1}^{\rho} \bigg(-\sin^2 \left(\pi \frac{s}{\nu_j}\right)\bigg)^{-r_j(k_j)}
\\& \times \prod_{j=1}^{\rho}
\prod_{\substack{\ell\in \{1, \ldots, \nu_j\}\\ \ell\neq k_j}} \bigg(\sin\left(\pi \left(\frac{-s-k+\ell}{\nu_j}\right)\right)
\sin \left(\pi \left(\frac{s-k+\ell}{\nu_j}\right)\right)\bigg)^{-r_j(\ell)}
\left(\frac{\sin(\pi s) \cos(\pi s)}{s}\right)^{\tau_0}.
\end{align*}
Recall that $k_j\in \{1, \ldots, \nu_j\}$ is such that $k_j\equiv k\bmod{\nu_j}$.
With this expression we conclude, upon taking $s\to 0$ and using \eqref{e:sum_rjk_int}, that
\begin{equation*}
{\rm ord}_{H_1}(0) = 2m(2g-2+\rho+\tau)-2\tau_0-2\sum_{j=1}^{\rho} r_j(k_j)
= 2\big(m(2g-2+\rho+\tau)-\tau_0 - \tilde{\tau}_0\big).
\end{equation*}
From \eqref{e:ordR^2_ordH1}, we have that ${\rm ord}_{R}(0) = \big(m(2g-2+\rho+\tau)-\tau_0 - \tilde{\tau}_0\big)-n_0$, so then
\begin{align*}
\lim_{s\to 0} \big(s^{-2{\rm ord}_R(0)}R(s)^2\big)
=& a_{n_0}^{-2} c(\chi, \Gamma)^{-2}
\\& \times  (-\pi)^{{\rm ord}_{H_1}(0)} \pi^{\tau_0}
\prod_{j=1}^{\rho}\nu_j^{2r_j(k_j)}
\prod_{j=1}^{\rho} \prod_{\substack{\ell\in \{1, \ldots, \nu_j\}\\ \ell\neq k_j}}
\sin\left(\pi \left(\frac{-k+\ell}{\nu_j}\right)\right)^{-2r_j(\ell)},
\end{align*}
which proves \eqref{e:Ruelle_limit_square_kint} and
completes the proof of Theorem \ref{thm:Ruelle_vanishing_constant}.

\subsubsection{The lead term at $s=0$ when $k=0$}\label{sss:leadatsequal0}

For all integers $k$, the Selberg zeta function $Z(s)$ has a zero at $s=1$ if and only if $0$ is an eigenvalue
of the weighted Laplacian. By \cite[p.371]{Hejhal83}, the $L^2$-eigenspace of the eigenvalue $0$ of the Laplacian
$\Delta_{2k}$ for any $k\in \mathbb{Z}$ is non-empty and one-dimensional.  Equivalently, we can say that
$Z(s)$ has a simple zero at $s=1$ for all $k\in \mathbb{Z}$.
As a result, we have that
\begin{equation*}
\lim_{s\to 0} \frac{Z(1-s)}{Z(1+s)} = -1.
\end{equation*}
By \eqref{e:Z-fe} and \eqref{eq: R funct equations 2}, we can write that
\begin{equation*}
R(s) = \frac{Z(s)}{Z(s+1)} = \frac{1}{\kappa(s)}\frac{Z(1-s)}{Z(1+s)}.
\end{equation*}
Thus ${\rm ord}_R(0) = -{\rm ord}_{\kappa}(0)$, and then
\begin{align}\label{eq. limit of kappa}\notag
\lim_{s\to 0} \big(s^{-{\rm ord}_R(0)} R(s)\big)
&= \lim_{s\to 0} \bigg(s^{-{\rm ord}_R(0)} \frac{1}{\kappa(s)}\bigg) \lim_{s\to 0} \bigg(\frac{Z(1-s)}{Z(1+s)}\bigg)
\\&= - \lim_{s\to 0} \bigg(s^{-{\rm ord}_R(0)} \frac{1}{\kappa(s)}\bigg).
\end{align}
Therefore, in order to determine the lead term at $s=0$ for integral values of $k$ it suffices to compute
the limit \eqref{eq. limit of kappa}.

When $k=0$,  we can evaluate \eqref{eq. limit of kappa} by using the following proposition.

\begin{proposition}\label{cor:lim_R_at0}
Assume $k=0$, so then $\chi$ is an $m$-dimensional unitary representation of $\Gamma$.
Then with notation as above we have that
\begin{align}\label{e:limR_k=0}\notag
\lim_{s\to 0} (s^{-{\rm ord}_R(0)} R(s))
&= (-1)^{\frac{1}{2}{\rm tr}(I_{\tau_0} -\Phi(\frac{1}{2}))+\tau_0+1}
(d(1)a_{n_0})^{-1}
2^{m(2g-2)}
\pi^{m(2g-2+\rho +\tau)-\frac{1}{2}\tau_0-\tilde{\tau}_0}
\\& \times \prod_{j=1}^\rho\bigg(\frac{\nu_j^{r_j}}{\prod_{\ell=1}^{\nu_j-1} \sin\left(\pi \left(\frac{\ell}{\nu_j}\right)\right)^{r_j(\ell)}}\bigg)
\prod_{j=1}^\tau \prod_{p=m_j+1}^{m} (\sin(\pi \beta_{jp}))^{-1}.
\end{align}
\end{proposition}

\begin{proof}
Assume $s\in \mathbb{C}\setminus \big((-\infty,0] \cup [1, +\infty)\big)$.
From \eqref{e:ZI_fe} in Proposition \ref{prop: funct eq sel zeta} we can write that
\begin{equation*}
\frac{Z_{I}(s)}{Z_{I}(1-s)}
= \exp\bigg\{\frac{m\omega(\mathcal{F})}{2\pi} \bigg((-1+2s) \log(2\pi)
- \log\frac{\Gamma\left(s\right)}{\Gamma\left(1-s\right)}-2 \log \frac{G(s)}{G(1-s)}
\bigg)\bigg\},
\end{equation*}
where we have used that $G(s+1) = \Gamma(s)G(s)$.  Upon using the reflection formula for Gamma functions, we get that
\begin{equation*}
\frac{Z_{I}(s)}{Z_{I}(1-s)}=  (2\pi)^{(-1+2s)\frac{m\omega(\mathcal{F})}{2\pi}}  \bigg(\frac{\sin(\pi s)}{\pi}\bigg)^{-\frac{m\omega(\mathcal{F})}{2\pi}}
\bigg(\frac{G(1+s)}{G(1-s)}\bigg)^{-\frac{2m\omega(\mathcal{F})}{2\pi}}.
\end{equation*}
Since \eqref{e:Zell_fe} is well-defined for $s\in \mathbb{C}\setminus \big((-\infty,0] \cup [1, +\infty)\big)$, we
can use the identity $\sin\pi\left(\frac{\nu_j-s}{\nu_j}\right) = \sin\pi\left(\frac{s}{\nu_j}\right) $
together with \eqref{e:Zell_fe}, to arrive at
\begin{equation*}
\frac{Z_{\rm ell}(s)}{Z_{\rm ell}(1-s)}
=  2^{m\sum_{j=1}^{\rho} (1-\frac{1}{\nu_j})} \sin\left(\pi s\right)^{-m\sum_{j=1}^{\rho} \frac{1}{\nu_j}}
\prod_{j=1}^{\rho} \sin\left(\pi \frac{s}{\nu_j}\right)^{r_j(0)} \prod_{\ell=1}^{\nu_j-1} \frac{\sin\left(\pi \left(\frac{-s+\ell}{\nu_j}\right)\right)^{\frac{\alpha_j(\ell)}{\nu_j}+r_j(\ell)}}{\sin\left(\pi \left(\frac{s+\ell}{\nu_j}\right)\right)^{\frac{\alpha_j(\ell)}{\nu_j} }},
\end{equation*}
and by \eqref{e:Zpar_fe}, for $s\in \mathbb{C}\setminus \big((-\infty,0] \cup [1, +\infty)\big)$ we can write
\[
\frac{Z_{{\rm par}, 1}(s)}{Z_{{\rm par}, 1}(1-s)} = 2^{-m\tau (2s-1)}
\prod_{j=1}^\tau \prod_{p=m_j+1}^{m} (\sin(\pi \beta_{jp}))^{1-2s}
(-1)^{\frac{1}{2}{\rm tr}(I_{\tau_0} -\Phi(\frac{1}{2}))}
\bigg(\frac{\Gamma\left(\frac{3}{2}-s\right)}{\Gamma\left(s+\frac{1}{2}\right)}\bigg)^{\tau_0}.
\]
Combining the formulas for the identity, elliptic and parabolic contribution we conclude,
for $s\in \mathbb{C}\setminus \big((-\infty, 0]\cup [1, +\infty)\big)$, that

\begin{align}\label{e:kappa_k=0}\notag
\kappa(s) &= \frac{Z_{I}(s)  Z_{\rm ell}(s) Z_{{\rm par}, 1}(s)}{Z_{I}(1-s)Z_{\rm ell}(1-s)   Z_{{\rm par}, 1}(1-s)}L(s) \tilde{\varphi}(s)
\\ & \notag= (2\pi)^{(-1+2s)\frac{m\omega(\mathcal{F})}{2\pi} }2^{-m\tau (2s-1)}  \pi^{\frac{m\omega(\mathcal{F})}{2\pi}} 2^{m\sum_{j=1}^{\rho} (1-\frac{1}{\nu_j})} \bigg(\frac{G(1+s)}{G(1-s)}\bigg)^{-2\frac{m\omega(\mathcal{F})}{2\pi}}
\\ & \notag\times \sin(\pi s)^{-m(\frac{\omega(\mathcal{F})}{2\pi}+\sum_{j=1}^{\rho} \frac{1}{\nu_j})+\tau_0}
\prod_{j=1}^{\rho} \sin\left(\pi \frac{s}{\nu_j}\right)^{r_j(0)} \prod_{\ell=1}^{\nu_j-1} \frac{\sin\left(\pi \left(\frac{-s+\ell}{\nu_j}\right)\right)^{\frac{\alpha_j(\ell)}{\nu_j}+r_j(\ell)}}{\sin\left(\pi \left(\frac{s+\ell}{\nu_j}\right)\right)^{\frac{\alpha_j(\ell)}{\nu_j} }}
\\&  \times
\prod_{j=1}^\tau \prod_{p=m_j+1}^{m} (\sin(\pi \beta_{jp}))^{1-2s}
(-1)^{\frac{1}{2}{\rm tr}(I_{\tau_0} -\Phi(\frac{1}{2}))+\tau_0}
\bigg(\frac{\Gamma\left(1-s\right)}{\cos(\pi s) \Gamma\left(s+\frac{1}{2}\right)} \bigg)^{\tau_0} e^{c_{1}s}d(1) \tilde{\varphi}(s).
\end{align}
By evaluating the limit as $s\to 0^+$ in \eqref{e:kappa_k=0}, we get that
\begin{align*}
\lim_{s\to 0} (s^{-{\rm ord}_R(0)} &R(s))
=- \bigg(\lim_{s\to 0^+} \big(s^{{\rm ord}_R(0)} \kappa(s) \big) \bigg)^{-1}
\\& = (-1)^{\frac{1}{2}{\rm tr}(I_{\tau_0} -\Phi(\frac{1}{2}))+\tau_0+1}
(d(1)a_{n_0})^{-1}
2^{m(2g-2)}
\pi^{m(2g-2+\rho +\tau)-\frac{1}{2}\tau_0-\sum_{j=1}^{\rho} r_j}
\\& \times \prod_{j=1}^\rho\bigg(\frac{\nu_j^{r_j}}{\prod_{\ell=1}^{\nu_j-1} \sin\left(\pi \left(\frac{\ell}{\nu_j}\right)\right)^{r_j(\ell)}}\bigg)
\prod_{j=1}^\tau \prod_{p=m_j+1}^{m} (\sin(\pi \beta_{jp}))^{-1}.
\end{align*}
Recall that $\sum_{j=1}^{\rho} r_j = \tilde{\tau}_0$.  With all this, the proof of the proposition is complete.
\end{proof}

\begin{remark}
Let us consider the case when
$k=0$ and $\chi$ is the trivial representation of dimension $m=1$.
As such, $\rho=\tilde{\tau}_0$ and $\tau_0=\tau$, so then
 the degree of the divisor of $R(s)$ at $s=0$ equals $2g-2-n_0$.
Furthermore, \eqref{e:limR_k=0} becomes
the statement that
\begin{equation}\label{eq: ruele limit s=0 teo case}
\lim_{s\to 0} \big(s^{-(2g-2-n_0)}R(s)\big)
= (-1)^{1 +\frac{1}{2}{\rm tr}(I_{\tau_0} -\Phi(\frac{1}{2}))+\tau }(d(1)a_{n_0})^{-1} (2\pi)^{(2g-2)}\pi^{\frac{\tau}{2}} \prod_{j=1}^{\rho}\nu_j.
\end{equation}
The scattering determinant can be written as
\begin{equation*}
\varphi(s)=\left(\frac{\Gamma\left(s-\frac{1}{2}\right)}{\Gamma(s)}\right)^{\tau}d(1)e^{c_1s} \left( a_{n_0}s^{n_0} + O(s^{n_0+1})\right)
\quad \text{ as $s \to 0$.}
\end{equation*}
Therefore, the degree of the divisor of $\varphi(s)$ at $s=0$ is $\tau+n_0$, and the lead term in the expansion of $\varphi(s)$ at $s=0$ is $(\Gamma(-\frac{1}{2}))^{\tau} d(1)a_{n_0}=(-2\sqrt{\pi})^{\tau}d(1)a_{n_0}$. Therefore, if we write
\begin{equation*}
\varphi(s)= \varphi_1(0)s^{\tau+n_0} + O(s^{\tau+n_0+1}) \quad \text{ as $s \to 0$,}
\end{equation*}
then \eqref{eq: ruele limit s=0 teo case} becomes
\begin{equation}\label{e:R0_teo}
\lim_{s\to 0}s^{-(2g-2+\tau-(\tau+n_0))}R(s)= (-1)^{1 +\frac{1}{2}{\rm tr}(I_{\tau_0} -\Phi(\frac{1}{2}))}(2\pi)^{2g-2+\tau}\varphi_1(0)^{-1}\prod_{j=1}^{\rho}\nu_j.
\end{equation}
The evaluation in \eqref{e:R0_teo} is
precisely \cite[Theorem 3.2]{Teo20}.
\end{remark}

\subsubsection{Variation in $k$}

\begin{remark}\label{rem: k to 0}
Let us consider the setting where the weight $2k$ is such that $k \in (0,1)$
and study the limiting behavior of \eqref{e:Ruelle_limit_square_knotint} as $k$ approaches zero.
In this case, ${\rm ord}_R(0) = -n_0$ and the right-hand side of \eqref{e:Ruelle_limit_square_knotint} is
\begin{align*}
\lim_{s\to 0}& \big(s^{2n_0} R(s)^2\big)
= a_{n_0}^{-2} d(1)^{-2} 2^{2m(2g-2)} \pi^{\tau_0}\prod_{j=1}^\tau \prod_{p=m_j+1}^{m} (\sin(\pi \beta_{jp}))^{-2}
\\& \times
(\sin(\pi k))^{2m(2g-2-\rho+\tau)-2\tau_0} \prod_{j=1}^{\rho} \sin\left(\pi\frac{k}{\nu_j}\right)^{-2r_j(0)}
\prod_{j=1}^{\rho} \prod_{\ell=1}^{\nu_j-1} \sin\left(\pi \left(\frac{-k+\ell}{\nu_j}\right)\right)^{-2r_j(\ell)}.
\end{align*}
In a slight abuse of notation, set
$$
\tilde{\tau}_0 = \sum_{j=1}^{\rho} r_j(0),
$$
which is not an elliptic degree of singularity for $k>0$, but as $k\downarrow 0$ can be viewed as such.
Also, in order to ease the notation, set
$$
\mathcal{R}_2(k):=\lim_{s\to 0} \big(s^{2n_0} R(s)^2\big),
$$
which we view as a function of $k$.
Then the order of vanishing of $\mathcal{R}_2(k)$ as $k\to 0$ is
\begin{equation*}
{\rm ord}_{\mathcal{R}_2}(0) = 2m(2g-2-\rho+\tau)-2\tau_0-2\tilde{\tau}_0.
\end{equation*}
Furthermore, from \eqref{e:Ruelle_limit_square_knotint} we get that
\begin{align}\label{eq:klimit}\notag
\lim_{k\to 0} \big(k^{-2m(2g-2-\rho+\tau)+2\tau_0+2\tilde{\tau}_0} &\mathcal{R}_2(k)\big)
= a_{n_0}^{-2} d(1)^{-2} 2^{2m(2g-2)} \pi^{2m(2g-2-\rho+\tau)-\tau_0-2\tilde{\tau_0}}
\\& \times \prod_{j=1}^\tau \prod_{p=m_j+1}^{m} (\sin(\pi \beta_{jp}))^{-2}
\prod_{j=1}^{\rho} \nu_j^{2r_j}
\prod_{j=1}^{\rho} \prod_{\ell=1}^{\nu_j-1} \sin\left(\pi \frac{\ell}{\nu_j}\right)^{-2r_j(\ell)}.
\end{align}
We find it very interesting that \eqref{eq:klimit}
is structurally similar to \eqref{e:Ruelle_limit_square_kint}, when taking into
account \eqref{e:c-RZfe2}.  In a sense, \eqref{eq:klimit} can be veiwed
as a type of (regularized) continuity result for $R(s; \chi_{k})$
when $s$ and $k$ each approach zero.  More precisely, we mean the following.

Let $R(s;\chi_k)$ be the Ruelle zeta function associated to an $m\times m$
unitary multiplier system $\chi_k$ of a weight $2k>0$ for  $2k\notin 2\mathbb{Z}$.  Assume that
$\Gamma$, the fixed Fuchsian group of the first kind, is such that for $j=1,\ldots \rho$, the system
$\chi_k(R_j)$ has eigenvalues $\exp\left(-\frac{2\pi i}{\nu_j}(k+\alpha_{jp})\right)$ for
$p=1,\ldots, m$; see   Section \ref{sec: mult systems}.  Let $R(s;\chi_0)$ denote the Ruelle
zeta function associated to an  $m$-dimensional unitary representation $\chi_0$ such that
$\chi_0(R_j)$ has eigenvalues $\exp\left(-\frac{2\pi i}{\nu_j}\alpha_{jp}\right)$ for $p=1,\ldots, m$ . Then
\begin{equation}\label{eq:regcontinuity}
\lim_{s\to 0}\left|\frac{R(s;\chi_0)}{s^{[m(2g-2+\rho+\tau) - \tau_0-\tilde{\tau}_0]-n_0}}\right|=
\lim_{k\rightarrow 0^{+}}\frac{1}{k^{[m(2g-2+\rho+\tau) - \tau_0-\tilde{\tau}_0]}}\left| \lim_{s\to 0} \frac{R(s;\chi_k)}{s^{-n_0}} \right|.
\end{equation}
In other words, when $k\rightarrow 0^{+}$, we can see that $\chi_k$ tends to $\chi_0$ in the sense that
their eigenvalues on elliptic generators of $\chi_{k}$ approach the corresponding information for $\chi_{0}$.
From \eqref{e:Ruelle_limit_square_knotint}, we can evaluate the
limiting behavior of the absolute value as $k\rightarrow 0^{+}$ of first term in the series expansion of
$R(s;\chi_0)$ at $s=0$. In doing so, we obtain \eqref{eq:regcontinuity} which seems to be a type of regularized
interchange of limits in $s$ and $k$.
\end{remark}

Yet another interesting observation regarding the dependence on the weight $2k$ can be made by comparing the factors of functional equations for the Ruelle zeta functions associated to conjugated $m\times m$ multiplier systems of weights $2k$ and $-2k$, for any admissible, nonzero weight. Namely, we have the following corollary.

\begin{corollary}
Let $\Gamma$ be co-compact (meaning $\tau=0$ and $\varphi\equiv 1$) and let $\chi_k,\chi_{-k}: \Gamma \to \mathcal{U}(V)$ be the two $m$-dimensional unitary multiplier systems with admissible weights $2k$ and $-2k$ such that $\chi_{-k}=\overline{\chi_k}$. Then, the Ruelle zeta functions twisted by $\chi_k$ and $\chi_{-k}$ satisfy the functional equation
\begin{equation}\label{eq: twisted by k and -k}
R(s;\chi_{-k})R(-s;\chi_{-k})=R(s;\chi_{k})R(-s;\chi_{k}).
\end{equation}
\end{corollary}
\begin{proof}
Our starting point is the functional equation \eqref{eq: ruella zeta funct eq} with $\varphi\equiv 1$ and the equation \eqref{e:H} in which the parabolic contribution, given as the product in the last line is identically one.
Let us denote by $H(s;\pm k)$ the factor of the functional equation \eqref{eq: ruella zeta funct eq} for the Ruelle zeta function twisted by $\chi_{\pm k}$. We claim that $H(s;k)=H(s;-k)$.

First, we observe that replacing $k$ with $-k$ in the identity contribution, given by the first line of the expression \eqref{e:H} leaves it unaltered, hence it remains to see whether the terms in the second line of \eqref{e:H}, stemming from the elliptic contribution will remain unaltered with the change of the multiplier weight.

With the notation as in Section \ref{sss:elliptic}, for each $j\in\{1,\ldots,\rho\}$, the eigenvalues of  $\chi_k(R_j)$ are of the form $\exp\left(-\frac{2\pi i}{\nu_j}(k+\alpha_{jp})\right)$, $\alpha_{jp}\in\{0,\ldots,\nu_j-1\}$, $p=1,\ldots,m$, while the eigenvalues of $\chi_{-k}(R_j)$ are of the form $\exp\left(-\frac{2\pi i}{\nu_j}(-k+(\nu_j-\alpha_{jp}))\right)$, $p=1,\ldots,m$. Then, it is straightforward to conclude that, for any $\ell\in\{0,\ldots, \nu_j-1\}$, the number $r^+(\ell)$ of values $\alpha_{jp}\in\{0,\ldots,\nu_j-1\}$ such that $\alpha_{jp}+\ell \equiv\, 0 \,\mathrm{mod}\, \nu_j$ equals the number $r^-(\nu_j-\ell)$ of values $\nu_j-\alpha_{jp}\in\{0,\ldots,\nu_j-1\}$ such that $\nu_j-\alpha_{jp}+\nu_j-\ell \equiv\, 0\, \mathrm{mod}\, \nu_j$.

Therefore, the value $r_j^+(\ell)$ in the elliptic contribution in \eqref{e:H} associated to $\chi_k$ equals the value $r_j^-(\nu_j -\ell)$ in the elliptic contribution associated to $\chi_{-k}$, which proves that the elliptic contributions in $H(s;k)$ and $H(s;-k)$ are equal and completes the proof that $H(s;k)=H(s;-k)$. This, combined with  \eqref{eq: ruella zeta funct eq} yields \eqref{eq: twisted by k and -k}.
\end{proof}

Certianly, these computations suggest that one should study $R(s;\chi_{k})$ as a function of the two variables $s$ and $k$, or $\chi_{k}$, perhaps as in \cite{Fa78} or \cite{Br85}, \cite{Br86a} and \cite{Br86b}.

\section{R-torsion of a Seifert fibered space with irreducible $\mathrm{SL}(n,\mathbb{C})$ representation} \label{sec: Kitano rep}

In this section we will combine Theorem \ref{thm:Ruelle_vanishing_constant}, which studies the Ruelle zeta function at $s=0$,
with the results from \cite{Ki94} and \cite{Ki96}, which computes R-torsion in certain cases.  In doing so, we verify further
identities which are predicted by the Fried conjecture.

Let $X$ be the Seifert fibered space given by the Seifert index $\{b,(o,g); (\nu_1,\beta_1),\ldots,(\nu_\rho,\beta_\rho)\} $,
Assume that the genus $g\geq 1$, and let  $\tilde{\rho}:\pi_1(X) \to \mathrm{SL}(m,\mathbb{C})$ be an irreducible acyclic representation.
From the notations and conventions in Section \ref{subs: irreducible kitano}, recall that
$$
\tilde{\rho}(h)=\lambda I_m
\,\,\,\,\,
\text{\rm where}
\,\,\,\,\,
\lambda=\exp(2\pi i (a/m))
\,\,\,\,\,
\text{\rm for some}
\,\,\,\,\, a\in \{1,\ldots,m-1\}.
$$
For $j=1,\ldots \rho$, let $e_{j,1},\ldots,e_{j,m}$ be the eigenvalues of $\tilde{\rho}(q_j)$.  Furthemore, assume that
$\mu_j,\alpha_j\in\mathbb{Z}$ satisfy $\alpha_j\nu_j - \beta_j\mu_j=-1$ for $0<\mu_j<\nu_j$.
With this setup, the main theorem of \cite{Ki96} proves in this case that the R-torsion $\tau(X; \tilde{\rho})$ is given by
\begin{equation}\label{eq:torsionfromKitano}
\tau(X; \tilde{\rho})=(\lambda-1)^{m(2-2g-\rho)}\prod_{j=1}^{\rho}\left(\lambda^{\alpha_j}e_{j,1}^{\mu_j}-1\right)\ldots \left(\lambda^{\alpha_j}e_{j,n}^{\mu_j}-1\right).
\end{equation}

If $\tilde{\rho}$ is odd-dimensional, then $\tau(X; \tilde{\rho})$ is well-defined up to a factor $\pm 1$.
Moreover, according to \cite[Lemma 2.3]{Ki96}, all eigenvalues of $\rho(q_j)$ are $m\nu_j$-th roots of unity for all $j=1,\ldots,\rho$.
Let $\tilde{s}_j$ denote the homotopy class of the core of the exceptional fibers $S_j$ with the index $(\nu_j,\beta_j)$ for each $j=1,\ldots,\rho$.
Then, as discussed in Section \ref{subs: irreducible kitano}, the eigenvalues of $\tilde{\rho}(\tilde{s}_j)$ are the numbers
$$
\lambda_{jp}:= \exp\left(-\frac{2\pi i}{\nu_j}(k+\alpha_{jp})\right)
\,\,\,\,\,
\text{\rm where}
\,\,\,\,\,
k=a/m
$$
and $\alpha_{jp}$ with  $p=1,\ldots, m$ are integers from the set $\{0,\ldots,\nu_j-1\}$.
On the other hand, the numbers $\lambda^{\alpha_j}e_{j,p}^{\mu_j}$ with $p=1,\ldots,m$ are eigenvalues of $\tilde{\rho}(q_j^{\mu_j}h^{\alpha_j})$
for any $j=1,\ldots,\rho$; hence are the eigenvalues of $\tilde{\rho}(\tilde{s}_j)$.
With all this, we can rewrite \eqref{eq:torsionfromKitano} as
\begin{equation}\label{eq: Reid torsion n dim rep}
\tau(X; \tilde{\rho})=(\lambda-1)^{m(2-2g-\rho)}\prod_{j=1}^{\rho}\prod_{p=1}^{m}\left(\exp\left(-\frac{2\pi i}{\nu_j}(k+\alpha_{jp})\right)-1\right).
\end{equation}

According to Lemma \ref{lem: existence of mult system}, we can associate to the representation $\tilde{\rho}$  an $m\times m $ unitary
multiplier system $\chi_{\tilde{\rho}}$ of weight $2k=2a/m$ such that for $j=1,\ldots,\rho$ the eigenvalues of
$\chi_{\tilde{\rho}}(R_j)$ coincide with the eigenvalues of $\tilde{\rho}(\tilde{s}_j)$.  Let us do so.

\begin{theorem} \label{thm: irreducible rep}
Let $X$ be the Seifert fibered space with Seifert index $\{b,(o,g); (\nu_1,\beta_1),\ldots,(\nu_\rho,\beta_\rho)\} $ and genus $g\geq 1$.
Let $\Gamma\subset \mathrm{SL}(2,\mathbb{R})$ be a genus $g$ Fuchsian group of the first kind which does not contain parabolic elements,
with $-I_2\in\Gamma$, and such that $\Gamma$ contains $\rho\geq 1$ classes of elliptic elements with representatives $R_1,\ldots,R_\rho$ of orders $\nu_1,\ldots,\nu_\rho$ respectively.
Let $\tilde{\rho}:\pi_1(X) \to \mathrm{SL}(m,\mathbb{C})$ be an irreducible acyclic representation such with eigenvalues described above.
Let $\chi_{\tilde{\rho}}$ be an $m$-dimensional unitary multiplier system on $\Gamma$ with weight $2k=2a/m$
and such that for all $j=1,\ldots,\rho$  the eigenvalues of $\chi_{\tilde{\rho}}(R_j)$ are
$\lambda_{jp}:= \exp\left(-\frac{2\pi i}{\nu_j}(k+\alpha_{jp})\right)$ for $p=1,\ldots, m$.
Then the Ruelle zeta function associated to $\Gamma$ and the multiplier system $\chi_{\tilde{\rho}}$ is non-vanishing at $s=0$ and
\begin{equation*}
|R(0;\chi_{\tilde{\rho}})|^{-1}=| \tau(X; \tilde{\rho})|.
\end{equation*}
\end{theorem}
\begin{proof}
We start by observing that $|e^{-2i\alpha}-1|=2|\sin\alpha|$. Hence, from \eqref{eq: Reid torsion n dim rep}, we have that
\begin{equation}\label{eq: Reid torsion n dim rep absolute value}
|\tau(X; \tilde{\rho})| = 2^{m(2-2g-\rho)}|\sin(k \pi)|^{m(2-2g-\rho)}\prod_{j=1}^{\rho}\prod_{p=1}^{m}2\left|\sin\left(\frac{\pi(k+\alpha_{jp})}{\nu_j}\right)\right|.
\end{equation}
Since $k\notin \mathbb{Z}$ and there are no parabolic elements, then $n_0=0$.  Also, because
there are no parabolic elements, $\tau=0$ and $\tau_0=0$.
From Theorem \ref{thm:Ruelle_vanishing_constant}, we have that $R(s; \chi)$ is non-vanishing at $s=0$
and, by \eqref{e:Ruelle_limit_square_knotint}, the lead term of $R(s; \chi)^2$ at $s=0$ is equal to
\begin{equation}\label{eq: R at zero precomp}
\lim_{s\to 0} R(s; \chi)^2
= 2^{2m(2g-2)}
(\sin(\pi k))^{2m(2g-2-\rho)} \prod_{j=1}^{\rho} \prod_{\ell=1}^{\nu_j} \sin\left(\frac{\pi(-k+\ell)}{\nu_j}\right)^{-2r_j(\ell)}.
\end{equation}

For all $j=1,\ldots,\rho$ and $\ell\in\{1,\ldots,\nu_j\}$ let $S_j(\ell)$ denote the set of all $\alpha_{jp}$ for
$p\in\{1,\ldots,n\}$ such that $\alpha_{jp}+\ell \equiv 0 \bmod{\nu_j}$.  Recall that $r_j(\ell)$ is the cardinality of $S_j(\ell)$.
Then for all $\alpha_{jp}\in S_{j}(\ell)$ we have that
\begin{equation}
\left|\sin\left(\frac{\pi(-k+\ell)}{\nu_j}\right)\right| = \left|\sin\left(\frac{\pi(k+\alpha_{jp})}{\nu_j}\right)  \right|.
\end{equation}
Since that sets $S_j(\ell)$ form the partition of the multiset of all $\alpha_{jp}$, when counting with multiplicities, we have that
\begin{equation}\label{eq: product of sines precomp}
\prod_{\ell=1}^{\nu_j} \left|\sin\left(\frac{\pi(-k+\ell)}{\nu_j}\right)\right|^{-r_j(\ell)}
= \prod_{p=1}^{m} \left|\sin\left(\frac{\pi(k+\alpha_{jp})}{\nu_j}\right)  \right|^{-1}.
\end{equation}
Substituting \eqref{eq: product of sines precomp} into \eqref{eq: R at zero precomp}, and then removing the square,
we get that $|R(0; \chi)|^{-1}$ coincides with \eqref{eq: Reid torsion n dim rep absolute value},  and
the proof of the theorem is complete.
\end{proof}


\section{R-torsion of a Seifert fibered space for certain $2N$-dimensional representations}\label{sec: Yam appl}

In \cite[Proposition 4.8]{Yamaguchi17} the author computes the so-called \emph{higher} R-torsion for a Seifert fibered space $X$ with the Seifert index $\{2-2g,(o,g); (\nu_1,\nu_1-1),\ldots,(\nu_\rho,\nu_\rho-1)\}$,
\footnote{Note that this index slightly differs from the one in \cite{Yamaguchi22}.
The notation is adopted to our convention for the presentation of the fundamental group so that presentations of the fundamental group in
\cite[Remark 2.1]{Yamaguchi22} and \eqref{eq: prez of fund group} with the above parameters coincide.}
and in \cite{Yamaguchi22} they proved that for certain values of the parameters in hand the Fried conjecture is true for this higher R-torsion.
In this section we will prove the Fried conjecture for a wider set of parameters than considered in \cite{Yamaguchi22}.
We refer to \cite{Yamaguchi17}
for the discussion which differentiates higher R-torsion from R-torsion.

Let $\rho_{2N}=\sigma_{2N}\circ\tilde{\rho}$ denote the acyclic representation as described in \cite{Yamaguchi22},
which is summarized in Section \ref{sebsect: yam represent}.  The multiplier system we will associate to this data
will have dimension $m=2N$ and weight $2k=1$, and its action on elliptic elements
of $\Gamma$ is described in Theorem \ref{prop: rad torsion eq ruelle zeta at 0};
Lemma \ref{lem: existence of mult system} ensures that such a multiplier system exists.

Let us recall \cite[Proposition 4.8]{Yamaguchi17}.  It is shown that the higher R-torsion of the aforementioned
Seifert fibered space $X$ with representation $\rho_{2N}$ is given by
\begin{equation}\label{eq: torsion f-la}
\tau(X;\rho_{2N})=2^{-2N(2-2g-\rho)}\prod_{j=1}^{\rho}\prod_{k=1}^{N}\left(2\sin\left(\frac{\pi(2k-1)\eta_j}{2\nu_j}\right)\right)^{-2}.
\end{equation}
Note that \cite[Lemma 4.2]{Yamaguchi17} shows that for all $j=1,\ldots,\rho$ the numbers $\eta_j$ are odd and coprime to $\nu_j$.

\begin{theorem}\label{prop: rad torsion eq ruelle zeta at 0}
Let $\Gamma$ be a subgroup of $\mathrm{SL}(2,\mathbb{R})$ such that $-I_2\in\Gamma$ and assume $\Gamma$ is co-compact of genus $g\geq 1$ with $\rho\geq 1$ classes of elliptic elements and with representatives $R_1,\ldots, R_{\rho}$.
Let $\chi$ be a multiplier system of weight $2k=1$ and even dimension $2N$ such that for each $j\in\{1,\ldots,\rho\}$, $\chi(R_j)$
has $2N$ eigenvalues of the form $\exp\left(\pm i\pi (2\ell-1)\eta_j/\nu_j\right)$ for $\ell=1,\ldots,N$.
Then, the higher R-torsion \eqref{eq: torsion f-la} satisfies the relation that
\begin{equation}
\tau(X;\rho_{2N}) =|R(0;\chi)|.
\end{equation}
\end{theorem}

\begin{proof}
For each $j=1,\ldots,\rho$, let
$$
S_{2N}(j)= \left\{ \frac{\pm(2l-1)\eta_j -1}{2}: l\in \{1,\ldots,N\}\right\}.
$$
For each $\ell=1,\ldots,\nu_j$, denote by $S_{2N}(j,\ell)$ the subset of $S_{2N}(j)$
 such that
 $$
 S_{2N}(j, \ell) = \left\{x\in S_{2N}(j):\, x+\ell \equiv 0\bmod{\nu_j}\right\}.
 $$
Let $r_j(\ell)$ denote the cardinality of $S_{2N}(j,\ell)$.
Obviously, for each $\ell=1,\ldots, \nu_j$, the sets $S_{2N}(j,\ell)$ form a partition of the set $S_{2N}(j)$ for each $j=1,\ldots,\rho$.

We start by expressing $\tau(X;\rho_{2N})$ given by \eqref{eq: torsion f-la} in a different form.  By using the identity
\begin{equation*}
\prod_{k=1}^{N}\sin \left(\frac{\pi(2k-1)\eta_j}{2\nu_j}\right)^{-1}
= (-1)^{N} \prod_{k=1}^{N} \sin \left(\frac{\pi(-(2k-1))\eta_j}{2\nu_j}\right)^{-1}
\end{equation*}
we get that
\begin{equation}\label{eq: torsion expr}
\tau(X;\rho_{2N})
=(-1)^N2^{2N(2g-2)}\prod_{j=1}^{\rho} \prod_{k=1}^{N}\sin \left(\frac{\pi(2k-1)\eta_j}{2\nu_j}\right)^{-1}
\sin \left(\frac{\pi(-(2k-1))\eta_j}{2\nu_j}\right)^{-1}.
\end{equation}
For an arbitrary, fixed $\ell\in\{1,\ldots, \nu_j\}$ assume that $k_1 \in \{1,\ldots,N\}$ is such that
$$
\frac{\pm(2k_1-1)\eta_j -1}{2} \in S_{2N}(j,\ell).
$$
Then $2\ell-1 \equiv\pm (2k_1-1)\eta_j\, (\, \text{mod }\, 2\nu_j)$,  hence
\begin{equation*}
\left|\sin \left(\frac{\pm\pi(2k_1-1)\eta_j}{2\nu_j}\right) \right| = \left|\sin \left(\frac{\pi(2\ell-1)}{2\nu_j} \right)\right| .
\end{equation*}
Moreover, for every $k\in\{1,\ldots,N\}$ there is a unique $\ell\in\{1,\ldots,\nu_j\}$ such that
$$
\frac{\pm(2k_1-1)\eta_j -1}{2} \in S_{2N}(j,\ell).
$$
Therefore, given that $S_{2N}(j,\ell)$ with $\ell=1,\ldots, \nu_j$ is a partition of the set $S_{2N}(j)$ for each $j=1,\ldots,\rho$,
\eqref{eq: torsion expr} becomes
\begin{equation}\label{eq: torsion expr final}
\tau(X;\rho_{2N})= 2^{2N(2g-2)}\prod_{j=1}^{\rho} \prod_{\ell=1}^{\nu_j}\left|\sin \left(\frac{\pi(2\ell-1)}{2\nu_j}\right)\right|^{-r_j(\ell)}
\end{equation}

On the other hand,  we know that in the notation of Section \ref{sec: mult systems} the eigenvalues
$\exp\left(\pm i\frac{\pi (2\ell-1)\eta_j}{\nu_j}\right)$
can be written as $\exp\left(- \frac{2i\pi}{\nu_j} \left(\frac{1}{2} + \alpha_{jp}\right)\right)$ for some integers $\alpha_{jp}\in\{0,\ldots,\nu_j-1\}$
with $p=1,\ldots,2N$.
It is trivial to observe that $\alpha_{jp}$ are then precisely the numbers from the set $S_{2N}(j,\ell)$ when reduced modulo $\nu_j$.
This proves that for any $\ell \in \{1,\ldots,\nu_j\}$ the number of solutions to the congruence $\alpha_{jp}+\ell \equiv 0\bmod{\nu_j}$
equals $r_j(\ell)$.
Consider Theorem \ref{thm:Ruelle_vanishing_constant}
when taking $k=\frac{1}{2}$, $\tau=0$, hence $\tau_0=n_0=0$,  and $m=2N$.  Indeed, we have that
$R(s; \chi)$ is not vanishing at $s=0$, and we have the equality that
\begin{equation*}\label{eq: r in terms of elliptics}
\lim_{s\to 0} R(s; \chi)^2
= 2^{4N(2g-2)} \prod_{j=1}^{\rho} \prod_{\ell=1}^{\nu_j} \sin\left(\pi \left(\frac{-1+2\ell}{2\nu_j}\right)\right)^{-2r_j(\ell)}.
\end{equation*}
Therefore, $\lim_{s\to 0} |R(s; \chi)|$ coincides with \eqref{eq: torsion expr final}, and
the proof of the theorem is completed.
\end{proof}

\begin{remark}
As stated, in \cite{Yamaguchi22}, the author considers a special case of the representation $\rho_{2N}$ defined in Section \ref{sebsect: yam represent}.
Specifically, it is assume in \cite{Yamaguchi22} that
$\eta_j=1$ for all $j=1,\ldots, \rho$, and in this case they prove Theorem \ref{prop: rad torsion eq ruelle zeta at 0}.
Our result in Theorem
\ref{prop: rad torsion eq ruelle zeta at 0} generalizes the main theorem
in \cite{Yamaguchi22} to the setting of all representations described in \cite{Yamaguchi17}.
\end{remark}

\section{Evaluating $R(0,\chi)$ for congruence subgroups and Dirichlet character}\label{s:congruence}

Beyond the Fried conjecture itself, which asserts a relation between R-torsion and special values of Ruelle zeta
function, one can consider special values of the Ruelle zeta function as a question of interest in and of itself.
Indeed, on \cite[p.162]{Fr86c}, Fried discusses the prospect of relating the Ruelle zeta function to
$L$-functions of number theory.  Going further, a considerable research presented \cite{CFM16} lies
in the visualizes $\vert R(0,\chi)\vert$ as a type of regulator, as posed in \cite{Fr86c}.

In this section, we will make explicit the results from Theorem \ref{thm:Ruelle_vanishing_constant}, and we will
evaluate the lead term of $\vert R(s,\chi)\vert$ at $s=0$ for various situations
where $\Gamma$ is a congruence subgroup $\Gamma_{0}(N)$ and the representation $\chi$ is associated to a Dirichlet character.
Let us describe the setting we investigate.

For a positive integer $N$, let $\Gamma=\Gamma_0(N)=\left\{\sm a & b\\ c& d\esm \in {\rm SL}_2(\mathbb{Z}):\, c\equiv 0\bmod{N}\right\}$
denote a congruence subgroup of level $N$.
Let $\chi$ be a unitary multiplier system of dimension one (i.e., $m=1$) and weight $2k=0$.
In particular, we take $\chi$ be a Dirichlet character modulo $N$.
The representation we consider is defined by $\chi(\gamma) = \chi(d)$ for
$\gamma=\sm a & b\\ c& d\esm \in\Gamma_0(N)$.
When $k=0$, our main result from Theorem \ref{thm:Ruelle_vanishing_constant} is that
\begin{multline}\label{e:limR_k=0_N}
\left|\lim_{s\to 0} (s^{-{\rm ord}_R(0)} R(s; \chi))\right|
= |d(1) a_{n_0}|^{-1}
2^{(2g-2)}
\pi^{(2g-2+\rho +\tau)-\frac{1}{2}\tau_0-\tilde{\tau}_0}
\\ \times \prod_{j=1}^\rho\bigg(\nu_j^{r_j} \prod_{\ell=1}^{\nu_j-1} \left|\sin\left(\pi \ell/\nu_j\right)\right|^{-r_j(\ell)}\bigg)
\prod_{j=1}^\tau \prod_{p=m_j+1}^1 |\sin(\pi \beta_{jp})|^{-1}.
\end{multline}
We now will specialize to certain congruence subgroups and Dirichlet characters and then compute all
quantities in \eqref{e:limR_k=0_N} explicitly in terms of number theoretic functions.
As predicted in \cite{Fr86c}, special values of Dirichlet $L$-functions will appear in this context,
specifically in the evaluations for $(d(1)a_{n_0})^{-1}$.

\subsection{Topological invariants in terms of the level} \label{sec:notation_N}
We will continue to use the notation from
Section \ref{sec:notation}.
The number of $\Gamma$-congruence classes of elliptic elements is $\rho$,
and we let $\{R_1, \ldots, R_\rho\}$ be representative of distinct $\Gamma$-conjugacy classes of elliptic elements.
As in Section \eqref{e:theta_nu_ellpitic},
$R_j$ is conjugate to $\theta_{\pi/\nu_j}\in {\rm SO}(2)$  for $j=1, \ldots, \rho$ and $\nu_j\in \mathbb{Z}_{\geq 2}$.
The order of the centralizer $Z(R_j)$ is  $2\nu_j$.
From \cite[Theorem 4.2.7]{miy89}, we have that
the only possible for values $\nu_j$ are either $\nu_j=2$ or $\nu_j=3$.
For a given $\nu\in \mathbb{Z}_{\geq 2}$,
\begin{equation}\label{eq: no all elements congr}
\#\left\{j\in \{1, \ldots, \rho\}:\, \nu_j=\nu \right\}
=\begin{cases}
1 & \text{ when } N=1\text{ and } \nu\in \{2, 3\}, \\
\prod_{p\mid N}\bigg(1+\left(\frac{-1}{p}\right)\bigg) & \text{ when } \nu=2 \text{ and } 4\nmid N, \\
\prod_{p\mid N}\bigg(1+\left(\frac{-3}{p}\right)\bigg) & \text{ when } \nu=3 \text{ and } 9\nmid N, \\
0 & \text{ otherwise. }
\end{cases}
\end{equation}
Since $\left(\frac{-1}{3}\right)=\left(\frac{-3}{2}\right)=-1$,
the number of $\Gamma_0(N)$-conjugacy classes of elliptic elements is given as
\begin{equation}\label{e:rhoN}
\rho(N) := \rho = \begin{cases}
2 & \text{ when } N=1, \\
\prod_{p\mid N}\bigg(1+\left(\frac{-1}{p}\right)\bigg)+\prod_{p\mid N}
\bigg(1+\left(\frac{-3}{p}\right)\bigg) & \text{ when } 4\nmid N\text{ and } 9\nmid N, \\
0 & \text{ otherwise. }
\end{cases}
\end{equation}
By \cite[Theorem 4.2.7]{miy89}, the number of inequivalent cusps for $\Gamma_0(N)$ is given by
\begin{equation}\label{e:tauN}
\tau(N):= \tau = \sum_{\delta\mid N} \varphi\left(\gcd(\delta, N/\delta)\right)
= \prod_{p\mid N} \bigg(p^{\lfloor\frac{{\rm ord}_p(N)}{2}\rfloor} + p^{\lfloor\frac{{\rm ord}_p(N)-1}{2}\rfloor}\bigg)>0.
\end{equation}
By \cite[(2.11)]{Iwa02}, the index of $\Gamma_0(N)$ in $\Gamma_0(1)={\rm SL}_2(\mathbb{Z})$ is
\begin{equation*}
[\Gamma_0(1): \Gamma_0(N)] = N\prod_{p\mid N} \left(1+p^{-1}\right).
\end{equation*}
It is well known that $\omega({\rm SL}_2(\mathbb{Z})\backslash \mathbb{H}) = \frac{\pi}{3}$,
from which we get that
\begin{equation}\label{e:vol_Gamma0N}
\omega(\Gamma_0(N)\backslash \mathbb{H}) = \frac{\pi}{3}N\prod_{p\mid N}(1+p^{-1}).
\end{equation}
Let $g(N)=g$ denote the genus of $\Gamma_0(N)\backslash \mathbb{H}$.
Then by using \eqref{eq: volume of F}, we write \eqref{e:vol_Gamma0N} in terms of topological data associated to $\Gamma_0(N)$,
namely that
\begin{equation}\label{e:vol_F_Gamma0N}
\frac{\omega(\Gamma_0(N)\backslash \mathbb{H})}{2\pi}
= \frac{1}{6}N\prod_{p\mid N}(1+p^{-1})
= 2g(N)-2 +\sum_{j=1}^{\rho(N)} (1-\nu_j^{-1}) + \tau(N).
\end{equation}

\subsection{Action of the character on elliptic and parabolic representatives}\label{sec: character action}
Let us now employ the notation from Section \ref{sec: mult systems}.
For each elliptic representative $R_j\in \Gamma_0(N)$ of order $\nu_j$ for $j=1,\ldots,\rho$,
 one has that  $\chi(R_j)^{\nu_j}=1$.  Hence, $\chi(R_j)$ has the form that
\begin{equation*}
\chi(R_j) = e^{-\frac{2\pi i}{\nu_j} \alpha_j} = \begin{cases} e^{-2\pi i \frac{\alpha_j}{2}} = (-1)^{\alpha_j} & \text{ when } \nu_j=2, \\
e^{-2\pi i \frac{\alpha_j}{3}} = \bigg(\frac{-1+\sqrt{3}}{2}\bigg)^{\alpha_j} & \text{ when } \nu_j=3.
\end{cases}
\end{equation*}
Since $m=1$,  $r_j=1$ if and only if $\chi(R_j)=1$, otherwise $r_j=0$.
Therefore the elliptic degree of singularity of $\chi$ is
\begin{equation}\label{e:tildetau0_N}
\tilde{\tau}_0 = \sum_{j=1}^{\rho} r_j = \# \left\{R_j: \, \chi(R_j)=1\right\}.
\end{equation}
The values $r_j(\ell)$ also can take only two possible values.  Either $r_j(\ell)=1$ when $\alpha_j+\ell \equiv 0\bmod{\nu_j}$,
or $r_j(\ell)=0$ in all other instances.  .

From \eqref{e:limR_k=0_N}, we have the following lemma.

\begin{lemma}\label{lem:limR_k=0_N_ell}
The elliptic contribution for $|\lim_{s\to 0} (s^{-{\rm ord}_R(0)} R(s; \chi))|$ is equal to
\begin{equation}\label{e:limR_k=0_N_ell}
\prod_{j=1}^\rho\left(\frac{\nu_j^{r_j}}{\prod_{\ell=1}^{\nu_j-1} \left|\sin\left(\pi \frac{\ell}{\nu_j}\right)\right|^{r_j(\ell)}}\right)
= 2^{C_{1}} 3^{C_{2}}
\left(\frac{\sqrt{3}}{2}\right)^{-C_{3}}.
\end{equation}
where
\begin{align}\notag
C_{1} :&= \#\{R_j:\, \nu_j=2 \text{ and } \chi(R_j)=1\},\,\, C_{2} := \#\{R_j:\, \nu_j=3\text{ and } \chi(R_j)=1\}, \\ &\text{  and  } \quad C_{3} := \#\{R_j:\, \nu_j=3 \text{ and } \chi(R_j)\neq 1\}. \label{eq:exponentof2}
\end{align}
\end{lemma}

\begin{proof}
We can get \eqref{e:limR_k=0_N_ell} by evaluating $\sin(\pi\ell/\nu_j)$ in the left-hand side of \eqref{e:limR_k=0_N_ell} for $\nu_j=2$ and $\nu_j=3$ and counting the number of each case appearing.
\end{proof}

Following \cite[\S2.5]{BLS20}, we can parameterise the parabolic representatives as the following.
Each cusp of $\Gamma=\Gamma_0(N)$ is equivalent to a cusp of the form
\[\frac{a}{c}
\,\,\,
\text{\rm for}
\,\,\, c\in \mathbb{Z}, \, c\geq 1, \, c\mid N, \, a\bmod{\gcd(c, N/c)},
\,\,\, \text{\rm and} \,\,\, \gcd(c, a)=1.
\]
By \cite[p.1089 and (2.45)]{BLS20}, we can fix a set $C_\Gamma$ containing exactly one representative $\frac{a}{c}$
which satisfies the above stated conditions, meaning that
\begin{equation*}
C_\Gamma = \left\{\frac{a}{c} :\, a, c\in \mathbb{Z}, \, c\geq 1, \, c\mid N, \, \gcd(a, c)=1\right\}/\sim_{\Gamma_0(N)}.
\end{equation*}
We make our choice of $C_\Gamma$ in such a way that
\begin{equation*}
\begin{cases}
\text{ for every $\frac{a}{c}\in C_\Gamma$ with $\gcd(c, N/c)>2$, we also have $\frac{-a}{c}\in C_\Gamma$}; \\
\text{ for every $c\mid N$ with $\gcd(c, N/c)\leq 2$, we have that $\frac{1}{c}\in C_\Gamma$}.
\end{cases}
\end{equation*}

For each such $\frac{a}{c}$, we fix a choice of scaling matrix by
\[W_{a/c} = \bpm a & b\\ c & d\ebpm \in {\rm SL}(2,\mathbb{Z}).\]
We also set
\[N_{a/c} = \bpm \sqrt{\frac{\gcd(c^2, N)}{N}} & 0 \\ 0 & \sqrt{\frac{N}{\gcd(c^2, N)}}\ebpm W_{a/c}^{-1}\]
and
\[S_{a/c} = N_{a/c}^{-1} \bpm 1 & -1\\ 0 & 1\ebpm N_{a/c} \in \Gamma_0(N), \]
which are parabolic representatives for the cusp $\frac{a}{c}$ of $\Gamma$.
Then
$$
\Gamma_{a/c} = \left\{\gamma\in \Gamma:\, \gamma \frac{a}{c} = \frac{a}{c}\right\} = \left<S_{a/c}\right>.
$$
Therefore, $m_{a/c}=1$ when $\chi(S_{a/c})=1$, meaning when $\frac{a}{c}$ is singular.
Note that in \cite[p.1076]{BLS20} the authors used the term \it open cusp \rm rather than singular.
Also, at this time we have indexed the cusps by $a/c$ and not by $j=1,\ldots, \tau(N)$ so that we are
consistent with the notation from \cite{BLS20}.  In particular,
the notation $m_j$ in Section \ref{sss:parabolic} corresponds to $m_{a/c}$ here, and
the degree of singularity of $\chi$ is then $\tau_0 = \sum_{a/c\in C_{\Gamma}}m_{a/c}$.

Let
\[\lambda_{a/c} = \chi(S_{a/c}) = e^{2\pi i \beta_{a/c}}, \quad \beta_{a/c}\in (0, 1].\]
By \cite[Lemma 2.9, (2.48)]{BLS20}, we write that
\begin{equation*}\label{e:chiSa/c}
\chi(S_{a/c}) = \chi\left(1-ac\frac{N}{\gcd(c^2, N)}\right)
= \chi\left(1-a\frac{N}{\gcd(c, N/c)}\right),
\end{equation*}
for $c\mid N$.
By \cite[Lemma 2.11]{BLS20}, $\frac{a}{c}$ is singular (or open) if and only if
$$
q={\rm cond}(\chi)\mid \frac{N}{\gcd(c, N/c)},
$$
as can be easily observed from the above discussion.
Furthermore by \cite[Lemma 2.11]{BLS20}, we have that
\begin{align}\label{e:tau0_N}\notag
\tau_0 &= \sum_{\substack{c\mid N\\ q\mid \frac{N}{\gcd(c, N/c)}}} \varphi(\gcd(c, N/c))
\\[3mm]&= \prod_{p\mid N}
\begin{cases} 2p^{{\rm ord}_p(N/q)} & \text{ if } {\rm ord}_p(N) < 2{\rm ord}_p(q), \\
p^{{\lfloor\frac{{\rm ord}_p(N)}{2}\rfloor}} + p^{\lfloor\frac{{\rm ord}_p(N)-1}{2}\rfloor} &\text{ otherwise. }
\end{cases}
\end{align}
With all this,  we conclude that the parabolic contribution
for \eqref{e:limR_k=0_N} can be written as
\begin{equation}\label{e:limR_k=0_N_parabolic}
\prod_{j=1}^{\tau} \prod_{p=m_j+1}^1 |\sin(\pi\beta_{jp})|^{-1}
= \prod_{\substack{c\mid N\\ \gcd(c, N/c)\nmid \frac{N}{{\rm cond}(\chi)}}} \prod_{\substack{a\bmod{\gcd(c, N/c)}\\ \gcd(a, c)=1}} 2 \left|1-\chi(S_{a/c})\right|^{-1}.
\end{equation}

\subsection{Scattering matrix}\label{ss:scatteringmtx_congruence}
We set the notation as above, meaning that $\Gamma_0(N)$ is a congruence subgroup of level $N$ and $\chi$ is a Dirichlet character modulo $N$ of conductor $q\mid N$, viewed as a multiplier system on $\Gamma_0(N)$ of weight $k=0$.
In this section, we follow \cite[\S2]{BLS20} and describe explicitly the determinant $\varphi(s;\chi)$ of the scattering matrix for $\Gamma_0(N)$ of weight $0$ with nebentypus $\chi$. Then, we compute the lead term in the Laurent series expansion of the Dirichlet series portion $\tilde{\varphi}(s;\chi)$ of $\varphi(s;\chi)$; see formulas \eqref{phiDirich}--\eqref{eqPhiB} with $k=0$.

We let $\xi_1$ and $\xi_2$ denote primitive Dirichlet characters satisfying $\xi_1(-1)=\xi_2(-1)$, with conductors $q_1$ and $q_2$ respectively.
We let
\begin{equation*}
F=\left\{(m, \xi_1, \xi_2): \, m\in \mathbb{Z}_{\geq 1}, \, mq_1\mid N, \, q_2\mid m, \, {\rm cond}(\chi\xi_2\overline{\xi_1})=1\right\}
\end{equation*}
and
\begin{equation*}
F_0 = \left\{(m, \xi_1, \xi_2)\in F:\, \xi_1=\overline{\xi_2}\right\}.
\end{equation*}
Note that ${\rm cond}(\chi\xi_2\overline{\xi_1})=1$ if and only if $\chi\xi_2(a)= \xi_1(a)$ for all $a\in (\mathbb{Z}/N\mathbb{Z})^\times$;
however, this does not imply $\chi\xi_2=\xi_1$ since the product $\chi\xi_2$ is not necessarily primitive.
For any $m\mid N$, we set
\begin{equation*}
F_m= \left\{(\xi_1, \xi_2):\, (m, \xi_1, \xi_2)\in F\right\}.
\end{equation*}
Let us write
\begin{equation*}
G_{N, q} = \left\{m\mid N:\, q\mid \frac{N}{\gcd(m, N/m)}\right\}.
\end{equation*}
By \cite[Lemma 2.21]{BLS20}, $F_m\neq \emptyset$ if and only if $m\in G_{N, q}$.
Moreover $\#F_m=\varphi(\gcd(m, N/m))$ and we have
\begin{equation*}
\#F = \sum_{m\in G_{N, q}} \# F_m = \sum_{m\in G_{N, q}}\varphi(\gcd(m, N/m)) = \tau_0.
\end{equation*}
For $\#F_0$, see \cite[Lemma 2.34]{BLS20}.
From \cite[(2.104)]{BLS20}, the determinant of the scattering matrix is given by
\begin{align*}
\varphi(s; \chi) &= (-1)^{\frac{\#F-\#F_0}{2}}
\bigg(\prod_{m\in G_{N, q}} \bigg(\frac{N}{\gcd(m, N/m)}\bigg)^{\varphi(\gcd(m, N/m))}\bigg)^{1-2s}
\bigg(\frac{\pi^{2s-1} \Gamma\left(1-s\right)}{\Gamma\left(s\right)}\bigg)^{\# F}
\\& \times \prod_{(m, \xi_1, \xi_2)\in F} \bigg(q_1^{1-2s} \frac{L(2-2s, \xi_1\xi_2\omega_m)}{L(2s, \xi_1\xi_2\omega_m)}\bigg),
\end{align*}
where $\omega_m$ denotes the trivial Dirichlet character modulo $m$.
We note that for each $(m, \xi_1, \xi_2)\in F$, $\xi_1\xi_2\omega_m$ is a Dirichlet character modulo $mq_1$.


For $\text{Re}(s) > 1$ we can write $\varphi(s; \chi) =L(s)\tilde{\varphi}(s; \chi)$ where
$
L(s) =\left( \frac{\Gamma \left( s-\frac{1}{2}
\right) }{\Gamma \left( s\right) }\right) ^{\tau_0}d(1) e^{c_{1}s}.
$
For our purposes, it is not necessary to compute the explicit values of $c_1$ and $d(1)$.
Note that $\# F = \tau_0$, so then we have
\begin{align}\label{e:tvarphi_Gamma0N}\notag
\tilde{\varphi}(s; \chi)
= &L(s)^{-1}d(1)^{-1} e^{-c_1s} \varphi(s; \chi)
\\ &\notag= (-1)^{\frac{\# F-\# F_0}{2}}
\bigg(\prod_{m\in G_{N, q}} \bigg(\frac{N}{\gcd(m, N/m)}\bigg)^{\varphi(\gcd(m, N/m))}\bigg)^{1-2s}
e^{-c_1s}d(1)^{-1}
\\& \times \bigg(\frac{\pi^{2s-1} \Gamma\left(1-s\right)}{\Gamma \left( s-\frac{1}{2}\right) }
\bigg)^{\# F}
\prod_{(m, \xi_1, \xi_2)\in F} \bigg(q_1^{1-2s} \frac{L(2-2s, \xi_1\xi_2\omega_m)}{L(2s, \xi_1\xi_2\omega_m)}\bigg).
\end{align}
Let us recall the Laurent expansion \eqref{eqPhiB}. We have the following lemma.

\begin{lemma}\label{lem:an0_N}
With the notation set above, we have
\begin{equation*}\label{e:n0_Gamma0N}
n_0 = - (\# F-\# F_0)-\sum_{(m, \xi_1, \xi_2)\in F} \sum_{\substack{p\mid mq_1 \\ (\xi_1\xi_2)_*(p)=1}} 1
\end{equation*}
and

\begin{align}\label{e:an0ec2_Gamma0N}\notag
a_{n_0} d(1)
&= (-1)^{\frac{\# F-\# F_0}{2}-(\#F-\#F_0)} (2\pi^{\frac{3}{2}})^{-\# F} (\pi^2 3^{-1})^{\#F_0}
\prod_{m\in G_{N, q}} \bigg(\frac{N}{\gcd(m, N/m)}\bigg)^{\varphi(\gcd(m, N/m))}
\\& \notag \times \prod_{(m, \xi_1, \xi_2)\in F} \bigg(\frac{{\rm cond}(\xi_1)}{{\rm cond}(\xi_1\xi_2)^{\frac{1}{2}}}
\frac{\prod_{p\mid mq_1}(1-(\xi_1\xi_2)_*(p)p^{-2})}{\prod_{p\mid mq_1, (\xi_1\xi_2)_*(p)\neq 1} (1-(\xi_1\xi_2)_*(p)) \prod_{p\mid mq_1, (\xi_1\xi_2)_*(p)=1}(-2\log p)}\bigg)
\\& \times \prod_{(m, \xi_1, \xi_2)\in F\setminus F_0}
\frac{L(2, (\xi_1\xi_2)_*)}{L(1, \overline{(\xi_1\xi_2)_*})}.
\end{align}
\end{lemma}

\begin{proof}
From \eqref{e:tvarphi_Gamma0N}, we first compute $n_0$ and take $s\to 0$ for $s^{-n_0} \tilde{\varphi}(s; \chi)$.

As one can see from  \eqref{e:tvarphi_Gamma0N},
the possible poles and zeros of $\tilde{\varphi}(s; \chi)$ only occur from the term involving the ratio of $L$-functions and Gamma functions. 
The ratio of Gamma functions has no zeros or poles around $s=0$.
For any $(m, \xi_1, \xi_2)\in F$, $L(2-2s, \xi_1\xi_2\omega_m)$ is analytic and has no zeros at $s=0$.
However $L(2s, \xi_1\xi_2\omega_m)$ has zeros at $s=0$ except in the case when $m=1$ and $\xi_1=\xi_2=1$.
In that case, $L(2s, \xi_1\xi_2\omega_m) = \zeta(2s)$ and $\zeta(0)=-\frac{1}{2}$.

More precisely, we have that
\begin{equation*}
L(2s, \xi_1\xi_2\omega_m)
= \bigg(\prod_{\substack{p\mid mq_1 \\ p\nmid {\rm cond}(\xi_1\xi_2)}} (1-(\xi_1\xi_2)_*(p)p^{-2s}) \bigg) L(2s, (\xi_1\xi_2)_*)
\end{equation*}
where $(\xi_1\xi_2)_*$ is the primitive character such that $\xi_1\xi_2$ is induced from and the product is taken over all primes $p$ satisfying stated conditions.
Note that by the assumption $\xi_1(-1)=\xi_2(-1)$, so then $\xi_1\xi_2$ is an even character.
By \cite[Theorem 4.15]{IK04}, we have the functional equation that
\begin{equation*}
L(2s, (\xi_1\xi_2)_*) = \left(\frac{{\rm cond}(\xi_1\xi_2))}{\pi}\right)^{\frac{1}{2}-2s} \frac{\Gamma\left(\frac{1}{2}-s\right)}{\Gamma\left(s\right)} L(1-2s, \overline{(\xi_1\xi_2)_*}).
\end{equation*}
When $(\xi_1\xi_2)_*=1$, then $L(0, (\xi_1\xi_2)_*) = \zeta(0) = -\frac{1}{2}.$
Otherwise,
\begin{equation*}
\lim_{s\to 0} \frac{L(2s, (\xi_1\xi_2)_*)}{s}
= ({\rm cond}(\xi_1\xi_2))^{\frac{1}{2}} L(1, \overline{(\xi_1\xi_2)_*}).
\end{equation*}
For a prime $p\mid mq_1$ with $p\nmid {\rm cond}(\xi_1\xi_2)$, if $(\xi_1\xi_2)_*(p)=1$ then
$(1-(\xi_1\xi_2)_*(p)p^{-2s}) = (1-p^{-2s})$ has a zero at $s=0$.
In particular, $$
\lim_{s\to 0} \frac{(1-p^{-2s})}{s} = -2\log p.$$

Recall that $F_{0} = \left\{(m, \xi_1, \xi_2)\in F:\, {\rm cond}(\xi_1\xi_2)=1\right\}$.
So we have that
\begin{equation*}
n_0 = \sum_{(m, \xi_1, \xi_2)\in F} \#\left\{p\mid mq_1:\, (\xi_1\xi_2)_*(p)=1 \right\} - (\#F-\#F_0).
\end{equation*}
Note that $n_0<0$.  Upon taking the limit $s\to 0$, we get
\begin{align*}
\lim_{s\to 0} (s^{-n_0} \tilde{\varphi}(s; \chi))
=&
 (-1)^{\frac{\# F-\# F_0}{2}}d(1)^{-1}(-2\pi^{\frac{3}{2}})^{-\# F}
\prod_{m\in G_{N, q}} \bigg(\frac{N}{\gcd(m, N/m)}\bigg)^{\varphi(\gcd(m, N/m))}
\\& \times \prod_{(m, \xi_1, \xi_2)\in F_0} \bigg(-2 q_1 \zeta(2) \prod_{p\mid mq_1}\bigg( \frac{1-p^{-2}}{-2\log p}\bigg)\bigg)
\\ \times \prod_{(m, \xi_1, \xi_2)\in F\setminus F_0}
\bigg( q_1&
\frac{\prod_{p\mid mq_1} (1-(\xi_1\xi_2)_*(p) p^{-2})}{\prod_{\substack{p\mid mq_1\\ (\xi_1\xi_2)_*(p)\neq 1}} (1-(\xi_1\xi_2)_*(p))
\prod_{\substack{p\mid mq_1\\ (\xi_1\xi_2)_*(p)=1}} (-2\log p)}
\frac{L(2, (\xi_1\xi_2)_*)}{({\rm cond}(\xi_1\xi_2))^{\frac{1}{2}} L(1, \overline{(\xi_1\xi_2)_*})} \bigg).
\end{align*}
From this identity, and using that $\zeta(2)=\frac{\pi^2}{6}$, we obtain \eqref{e:an0ec2_Gamma0N}.
\end{proof}

\begin{example}\rm
Assume that $\chi$ is the trivial character modulo $N$, so $q=1$.
Then
\begin{equation*}
F = \left\{(m, \xi_1, \xi_1): \, m\in \mathbb{Z}_{\geq 1},\, mq_1\mid N, \, q_1\mid m \right\}
= \left\{(mq_1, \xi_1, \xi_1):\, m\in \mathbb{Z}_{\geq 1}, \, mq_1^2\mid N\right\}.
\end{equation*}
and
\begin{equation*}
F_0 = \left\{(mq_1, \xi_1, \xi_1):\, m\in \mathbb{Z}_{\geq 1}, \, mq_1^2\mid N, \, \xi_1^2=1\right\}.
\end{equation*}
If we further assume that $N$ is square-free, then $F = F_0=\{(1, 1, 1)\}$ and $G_{N,1}$ is the set of divisors of $N$.
In this case formula \eqref{e:an0ec2_Gamma0N} reduces to
\begin{equation}\label{eq: an0c2 N sqfree}
a_{n_0}d(1)= \frac{\sqrt{\pi}}{6}N^{2^{\omega(N)}},
\end{equation}
where $\omega(N)$ is the number of prime divisors of $N$.
\end{example}

\subsection{The expansion of $R(s;\chi)$ at $s=0$}

In this section we compute the absolute value of the leading term in the expansion at $s=0$ of the Ruelle zeta function $R(s;\chi)$ associated to the congruence group $\Gamma_0(N)$ and the Dirichlet character modulo $N$, viewed as a multiplier system of weight zero.
This amounts to computing the right-hand side of \eqref{e:limR_k=0_N}.

We combine Lemma \ref{lem:limR_k=0_N_ell} and \eqref{e:limR_k=0_N_parabolic}
\begin{proposition}\label{prop:limR_k=0_N}
Assume $\Gamma=\Gamma_0(N)$, $\chi$ is a Dirichlet character modulo $N$ with ${\rm cond}(\chi)=q$ and $k=0$.
Then we have, with the notations in Sections \ref{sec:notation_N} and \ref{sec: character action} that
\begin{equation*}
{\rm ord}_{R}(0) = (\tau_0-\# F_0) + \sum_{(m, \xi_1, \xi_2)\in F} \sum_{\substack{p\mid mq_1 \\ (\xi_1\xi_2)_*(p)=1}} 1
+ 2g-2+\rho+\tau-\tau_0-\tilde{\tau}_0
\end{equation*}
and
\begin{equation*}\label{e:limR_k=0_N_formula}
\left|\lim_{s\to 0} (s^{-{\rm ord}_R(0)} R(s; \chi))\right|
= |d(1) a_{n_0}|^{-1}
2^{(2g-2)} \pi^{(2g-2+\rho +\tau)-\frac{1}{2}\tau_0-\tilde{\tau}_0}
E(N; \chi) P(N; \chi)
\end{equation*}
where
\begin{equation*}\label{e:ENchi_def}
E(N; \chi)= 2^{C_{1}} 3^{C_{2}}\left(\frac{\sqrt{3}}{2}\right)^{-C_{3}}
\end{equation*}
and
\begin{equation}\label{e:PNchi_def}
P(N; \chi)
= \prod_{\substack{c\mid N\\ \gcd(c, N/c)\nmid \frac{N}{{\rm cond}(\chi)}}} \prod_{\substack{a\bmod{\gcd(c, N/c)}\\ \gcd(a, c)=1}} 2 \left|1-\chi(S_{a/c})\right|^{-1}.
\end{equation}
Here we also have
\begin{align}\label{e:an0_abs_N}\notag
|a_{n_0}d(1)|
&= (2\pi^{\frac{3}{2}})^{-\# F} (\pi^2 3^{-1})^{\#F_0}
\prod_{m\in G_{N, q}} \bigg(\frac{N}{\gcd(m, N/m)}\bigg)^{\varphi(\gcd(m, N/m))}
\\& \notag \times \prod_{(m, \xi_1, \xi_2)\in F} \bigg(\frac{{\rm cond}(\xi_1)}{{\rm cond}(\xi_1\xi_2)^{\frac{1}{2}}}
\frac{\prod_{p\mid mq_1}(|1-(\xi_1\xi_2)_*(p)p^{-2}|)}{\prod_{p\mid mq_1, (\xi_1\xi_2)_*(p)\neq 1} |1-(\xi_1\xi_2)_*(p)| \prod_{p\mid mq_1, (\xi_1\xi_2)_*(p)=1}2\log p)}\bigg)
\\& \times \prod_{(m, \xi_1, \xi_2)\in F\setminus F_0}
\frac{\left|L(2, (\xi_1\xi_2)_*)\right|}{\left|L(1, \overline{(\xi_1\xi_2)_*})\right|}.
\end{align}
\end{proposition}

\begin{remark}
One can find formulas for $\rho$, $\tau$, $\tau_0$, $\tilde{\tau}_0$ and $g$ in \eqref{e:rhoN}, \eqref{e:tauN}, \eqref{e:tau0_N}, \eqref{e:tildetau0_N}, and \eqref{e:vol_F_Gamma0N} respectively.  The constants $C_{1}$, $C_{2}$ and $C_{3}$ are defined
in \eqref{eq:exponentof2}.
The number of elements in $F_0$ can be found in  \cite[Lemma 2.34]{BLS20}.
\end{remark}


For example, when $N=1$, then $\chi=1$, $\rho=2$, $\tau=\tau_0=1=\#F=\#F_0$, $\tilde{\tau_0}=\rho=2$ and
$|a_{n_0} d(1)|=\frac{\sqrt{\pi}}{6}$ which is given by \eqref{eq: an0c2 N sqfree}.
 We can also compute $E(N)$ in the view of formula \eqref{eq: no all elements congr}.
Then ${\rm ord}_R(0) = 2g-2 = -2$ and
\begin{equation*}
\left|\lim_{s\to 0} (s^{-{\rm ord}_R(0)} R(s; \chi))\right|
= \frac{9}{\pi^2}.
\end{equation*}
This evaluation agree with the result stated on \cite[p.79]{Teo20}.

\subsection{Specialization to prime square levels}

Let us now further specialise the computations of Proposition \ref{prop:limR_k=0_N}
to the case when the level $N=\ell^2$ is square of a prime number $\ell\geq 5$.
In doing so, we obtain the following result.

\begin{corollary}
Let $\ell \geq 5$ be a prime, $N=\ell^2$ and $\chi$ be a Dirichlet character modulo $N$ of conductor $\ell^b$ for $b\in \{0, 1, 2\}$.
Assume that the weight $2k=0$.
Then the order of the divisor of Ruelle zeta function $R(s; \chi)$ at $s=0$
is
\begin{equation}\label{e:ordR0_ell2}
{\rm ord}_R(0)
= -2+b -\tilde{\tau}_0+ \frac{(\ell-6)(\ell+1)}{6}
+ \begin{cases}
\frac{5}{3} & \text{ if } \ell\equiv 1\bmod{12}, \\
1 & \text{ if } \ell \equiv 5\bmod{12}, \\
\frac{2}{3} & \text{ if } \ell \equiv 7 \bmod{12}, \\
0 & \text{ if } \ell \equiv 11\bmod{12}.
\end{cases}
\end{equation}
Furthermore,
\begin{equation}
\left|\lim_{s\to 0} (s^{-{\rm ord}_R(0)} R(s; \chi))\right|
= 2^{\tau_0+c_1(\ell)} 3^{4-2b}  \pi^{\tau_0 -3(2-b) + {\rm ord}_R(0)} A_1(b)^{-1}A_2(b)^{-1}A_3(b)^{-1} E(\ell^2; \chi) P(\ell^2; \chi),
\end{equation}
where
\begin{equation*}
\tau_0 = \begin{cases}\ell+1& \text{ if } b\in \{0, 1\}, \\ 2 & \text{ if } b=2,
\end{cases}
\end{equation*}
\begin{equation*}
c_1(\ell) =  \frac{(\ell-6)(\ell+1)}{6}
- \begin{cases} \frac{7}{3} & \text{ if } \ell \equiv 1\bmod{12}, \\
1 & \text{ if } \ell\equiv 5 \bmod{12}, \\
\frac{4}{3} & \text{ if } \ell \equiv 7\bmod{12}, \\
0 & \text{ if } \ell \equiv 11\bmod{12},
\end{cases}
\end{equation*}
\begin{equation*}\label{e:Pell2}
P(\ell^2; \chi) = \begin{cases}
1 & \text{ if } b\in \{0, 1\}, \\
\prod_{\substack{a\bmod{\ell}\\ \gcd(a, \ell)=1}} 2\left|1-\chi(1-a\ell)\right|^{-1}
& \text{ if } b=2,
\end{cases}
\end{equation*}
\begin{equation*}
A_1(b)
= \begin{cases}
\ell^{4+\varphi(\ell)} & \text{ when } b\in \{0, 1\}, \\
\ell^4 & \text{ when } b=2,
\end{cases} \quad \quad
A_2(b) = \begin{cases}
\ell^{1+\frac{\ell-3}{2}}
\frac{(1-\ell^{-2})^3}{(2\log \ell)^3}
& \text{ if } b=0, \\
\ell^{\frac{3}{2}+\frac{\ell-5}{2}}
\frac{(1-\ell^{-2})^2}{(2\log\ell)^2}
& \text{ if } b=1, \\
1 & \text{ if } b=2
\end{cases}
\end{equation*}
and
\begin{equation*}
A_3 (b)
= \begin{cases}
\prod_{\xi\bmod{\ell}, \xi^2\neq 1} \frac{|L(2, \xi^2)|}{|L(1, \xi^2)|} & \text{ if } b=0, \\
\frac{|L(2, \chi)|^2}{|L(1, \chi)|^2}
\prod_{\xi\bmod{\ell}, \xi^2\neq \bar{\chi}} \frac{|L(2, \chi\xi^2)|}{|L(1, \chi\xi^2)|} & \text{ if } b=1, \\
\frac{|L(2, \chi)|^2}{|L(1, \chi)|^2} & \text{ if } b=2.
\end{cases}
\end{equation*}
Here $\tilde{\tau}_0 = \#\left\{R_j:\, \chi(R_j)=1\right\}$
and $E(\ell^2, \chi)$ is given as in Proposition \ref{prop:limR_k=0_N},
\end{corollary}

\begin{proof}
As stated, assume that
$N=\ell^2$ for a prime $\ell$ and $q={\rm cond}(\chi)=\ell^b$ for some $b\in \{0, 1, 2\}$.
Then
\begin{equation*}
\rho = \begin{cases}
\left(1+\left(\frac{-1}{\ell}\right)\right)+\left(1+\left(\frac{-3}{\ell}\right)\right) & \text{ when } \ell\notin\{2, 3\}, \\
0 & \text{ otherwise.}
\end{cases}
\end{equation*}
Note that $\left(\frac{-1}{\ell}\right)=(-1)^{\frac{\ell-1}{2}}$ and $\left(\frac{-3}{\ell}\right)=\left(\frac{\ell}{3}\right)$.
We then have that
\begin{equation*}
\rho = \begin{cases}
4 & \text{ if } \ell\equiv 1\bmod{12}, \\
2 & \text{ if } \ell \equiv 5, 7 \bmod{12}, \\
0 & \text{ if } \ell\equiv 11 \bmod{12}.
\end{cases}
\end{equation*}
We also have that $\tau=\ell+1$,
$$
\tau_0 = \begin{cases} \ell+1 & \text{ if } b\in \{0, 1\}, \\ 2 & \text{ if } b=2 \end{cases}
$$
and
\begin{equation*}
2g-2
= \frac{(\ell-6)(\ell+1)}{6}
- \begin{cases} \frac{7}{3} & \text{ if } \ell \equiv 1\bmod{12}, \\
1 & \text{ if } \ell\equiv 5 \bmod{12}, \\
\frac{4}{3} & \text{ if } \ell \equiv 7\bmod{12}, \\
0 & \text{ if } \ell \equiv 11\bmod{12}.
\end{cases}
\end{equation*}
Let us now compute ${\rm ord}_R(0)$ and $|a_{n_0} d(1)|$ which are given in \eqref{e:an0_abs_N}.

To begin, we get that
\begin{equation*}
F 
= \begin{cases}
\left\{(1, 1, 1)\right\} \cup \left\{(\ell, \xi_1, \xi_1): \xi_1\text{ primitive, }{\rm cond}(\xi_1)\in \{1, \ell\}\right\}
\cup\left\{(\ell^2, 1, 1)\right\} & \text{ when } b=0, \\
\left\{(1, \chi, 1)\right\} \cup \left\{(\ell, (\chi\xi_2)_*,\xi_2):\, \xi_2 \text{ primitive, }{\rm cond}(\xi_2)\in \{1,\ell\}\right\}
\cup\left\{(\ell^2, 1, \overline{\chi})\right\}
& \text{ when } b=1, \\
\left\{(1, \chi, 1)\right\}\cup \left\{(\ell^2, 1, \overline{\chi}) \right\}& \text{ when } b=2,
\end{cases}
\end{equation*}
and
\begin{equation*}
F_ 0 =\begin{cases}
\left\{(1, 1, 1), (\ell, \left(\frac{\cdot}{\ell}\right), \left(\frac{\cdot}{\ell}\right)), (\ell, 1, 1), (\ell^2, 1, 1)\right\}  & \text{ when } b=0, \\
\left\{(\ell, \xi, \bar{\xi}):\, {\rm cond}(\xi)=\ell, \xi^2=\chi \right\} & \text{ when } b=1, \\
\emptyset & \text{ when } b=2.
\end{cases}
\end{equation*}
Note that $\#F_0 = 2$ when $b=1$ since we have two choices, namely an even character $\xi$ or an odd character $\xi$.
Then
\begin{equation*}
\sum_{(m, \xi_1, \xi_2)\in F} \sum_{\substack{p\mid mq_1\\ (\xi_1\xi_2)_*(p)=1}} 1
= \begin{cases} 2& \text{ when } b=0, \\
1 & \text{ when } b=1, \\
0 & \text{ when } b=2. \end{cases}
= 2-b
\end{equation*}
When combining with the above formulas for $2g-2$, $\rho$, $\tau$ and $\#F_0$, we can get \eqref{e:ordR0_ell2}.
We also have that
\begin{equation*}
G_{\ell^2, \ell^b} 
= \begin{cases}
\left\{1, \ell, \ell^2\right\} & \text{ when } b=0, 1, \\
\left\{1, \ell^2\right\} & \text{ when } b=2.
\end{cases}
\end{equation*}
Then, when applying to \eqref{e:an0_abs_N}, we have that
\begin{equation*}
|a_{n_0} d(1)|
= (2\pi^{\frac{3}{2}})^{-\# F} (\pi^2 3^{-1})^{\#F_0}
A_1(b) A_2(b) A_3(b)
\end{equation*}
where
\begin{align*}
A_1(b) = \prod_{m\in G_{N, q}} \bigg(\frac{N}{\gcd(m, N/m)}\bigg)^{\varphi(\gcd(m, N/m))}
= \begin{cases}
\ell^{4+\varphi(\ell)} & \text{ when } b\in \{0, 1\}, \\
\ell^4 & \text{ when } b=2,
\end{cases}
\end{align*}
with
\begin{align*}
A_2(b) &= \prod_{(m, \xi_1, \xi_2)\in F} \bigg(\frac{{\rm cond}(\xi_1)}{{\rm cond}(\xi_1\xi_2)^{\frac{1}{2}}}
\frac{\prod_{p\mid mq_1}(|1-(\xi_1\xi_2)_*(p)p^{-2}|)}{\prod_{p\mid mq_1, (\xi_1\xi_2)_*(p)\neq 1} |1-(\xi_1\xi_2)_*(p)| \prod_{p\mid mq_1, (\xi_1\xi_2)_*(p)=1}2\log p)}\bigg)
\\ &= \begin{cases}
\ell^{1+\frac{\ell-3}{2}}
\frac{(1-\ell^{-2})^3}{(2\log \ell)^3}
& \text{ if } b=0, \\
\ell^{\frac{3}{2}+\frac{\ell-5}{2}}
\frac{(1-\ell^{-2})^2}{(2\log\ell)^2}
& \text{ if } b=1, \\
1 & \text{ if } b=2
\end{cases}
\end{align*}
and
\begin{align*}
A_3 (b) = \prod_{(m, \xi_1, \xi_2)\in F\setminus F_0}
\frac{\left|L(2, (\xi_1\xi_2)_*)\right|}{\left|L(1, \overline{(\xi_1\xi_2)_*})\right|}
= \begin{cases}
\prod_{\xi\bmod{\ell}, \xi^2\neq 1} \frac{|L(2, \xi^2)|}{|L(1, \xi^2)|} & \text{ if } b=0, \\
\frac{|L(2, \chi)|^2}{|L(1, \chi)|^2}
\prod_{\xi\bmod{\ell}, \xi^2\neq \bar{\chi}} \frac{|L(2, \chi\xi^2)|}{|L(1, \chi\xi^2)|} & \text{ if } b=1, \\
\frac{|L(2, \chi)|^2}{|L(1, \chi)|^2} & \text{ if } b=2.
\end{cases}
\end{align*}
Finally, from \eqref{e:PNchi_def} with $N=\ell^2$ and $\ell$ is a prime, we get that
\begin{equation*}
P(N; \chi) = \begin{cases}
1 & \text{ if } b\in \{0, 1\}, \\
\prod_{\substack{a\bmod{\ell}\\ \gcd(a, \ell)=1}} 2\left|1-\chi(S_{a/\ell})\right|^{-1}
= \prod_{\substack{a\bmod{\ell}\\ \gcd(a, \ell)=1}} 2\left|1-\chi(1-a\ell)\right|^{-1}
& \text{ if } b=2.
\end{cases}
\end{equation*}

Applying all of the above computations to Proposition \ref{prop:limR_k=0_N}, we the proof of the corollary is completed.

\end{proof}

\thispagestyle{empty}
{\footnotesize
\bibliographystyle{amsalpha}
\bibliography{reference_Ruelle_Sept}
}

\vspace{3mm}
\noindent
Jay Jorgenson \\
 Department of Mathematics \\
 The City College of New York \\
 Convent Avenue at 138th Street \\
 New York, NY 10031 U.S.A. \\
 e-mail: jjorgenson@mindspring.com

 \vspace{3mm}
\noindent
Min Lee \\
 Department of Mathematics \\
University of Bristol\\
 e-mail: min.lee@bristol.ac.uk

 \vspace{3mm}
\noindent
Lejla Smajlovi\'{c} \\
 Department of Mathematics and Computer Sciences\\
 University of Sarajevo\\
 Zmaja od Bosne 35, 71 000 Sarajevo\\
 Bosnia and Herzegovina\\
 e-mail: lejlas@pmf.unsa.ba

\end{document}